\documentclass[11pt,letterpaper]{amsart}
\usepackage[utf8]{inputenc}
\usepackage{kantlipsum} 
\usepackage{comment}
\calclayout
\newif\ifdraft
\draftfalse
\usepackage{amsmath,amsfonts,amsthm,mathrsfs}
\usepackage{amssymb, stmaryrd}

\usepackage[unicode,bookmarks,pagebackref]{hyperref}
\usepackage[dvipsnames]{xcolor}
\hypersetup{colorlinks=true,citecolor=NavyBlue,linkcolor=Brown,urlcolor=Orange}

\usepackage[alphabetic]{amsrefs}

\usepackage{enumitem}

\usepackage{chngcntr}

\ifdraft
\usepackage[notcite,notref,color]{showkeys}

\definecolor{labelkey}{gray}{0.5}
\fi

\usepackage{tikz}

\usepackage{tikz-cd}

\usetikzlibrary{matrix,arrows}
\newlength{\myarrowsize} 

\pgfarrowsdeclare{cmto}{cmto}{
	\pgfsetdash{}{0pt} 
	\pgfsetbeveljoin 
	\pgfsetroundcap 
	\setlength{\myarrowsize}{0.6pt}
	\addtolength{\myarrowsize}{.5\pgflinewidth}
	\pgfarrowsleftextend{-4\myarrowsize-.5\pgflinewidth} 
	\pgfarrowsrightextend{.8\pgflinewidth}
}{
	\setlength{\myarrowsize}{0.6pt} 
  	\addtolength{\myarrowsize}{.5\pgflinewidth}  
	\pgfsetlinewidth{0.5\pgflinewidth}
	\pgfsetroundjoin
	\pgfpathmoveto{\pgfpoint{1.5\pgflinewidth}{0}}
	\pgfpatharc{-109}{-170}{4\myarrowsize}
	\pgfpatharc{10}{189}{0.58\pgflinewidth and 0.2\pgflinewidth}
	\pgfpatharc{-170}{-115}{4\myarrowsize+\pgflinewidth}
	\pgfpathclose
	\pgfusepathqfillstroke
	\pgfpathmoveto{\pgfpoint{1.5\pgflinewidth}{0}}
	\pgfpatharc{109}{170}{4\myarrowsize}
	\pgfpatharc{-10}{-189}{0.58\pgflinewidth and 0.2\pgflinewidth}
	\pgfpatharc{170}{115}{4\myarrowsize+\pgflinewidth}
	\pgfpathclose
	\pgfusepathqfillstroke
	\pgfsetlinewidth{2\pgflinewidth}
}

\pgfarrowsdeclare{cmonto}{cmonto}{
	\pgfsetdash{}{0pt} 
	\pgfsetbeveljoin 
	\pgfsetroundcap 
	\setlength{\myarrowsize}{0.6pt}
	\addtolength{\myarrowsize}{.5\pgflinewidth}
	\pgfarrowsleftextend{-4\myarrowsize-.5\pgflinewidth} 
	\pgfarrowsrightextend{.8\pgflinewidth}
}{
	\setlength{\myarrowsize}{0.6pt} 
  	\addtolength{\myarrowsize}{.5\pgflinewidth}  
	\pgfsetlinewidth{0.5\pgflinewidth}
	\pgfsetroundjoin
	\pgfpathmoveto{\pgfpoint{1.5\pgflinewidth}{0}}
	\pgfpatharc{-109}{-170}{4\myarrowsize}
	\pgfpatharc{10}{189}{0.58\pgflinewidth and 0.2\pgflinewidth}
	\pgfpatharc{-170}{-115}{4\myarrowsize+\pgflinewidth}
	\pgfpathclose
	\pgfusepathqfillstroke
	\pgfpathmoveto{\pgfpoint{1.5\pgflinewidth}{0}}
	\pgfpatharc{109}{170}{4\myarrowsize}
	\pgfpatharc{-10}{-189}{0.58\pgflinewidth and 0.2\pgflinewidth}
	\pgfpatharc{170}{115}{4\myarrowsize+\pgflinewidth}
	\pgfpathclose
	\pgfusepathqfillstroke
	\pgfpathmoveto{\pgfpoint{1.5\pgflinewidth-0.3em}{0}}
	\pgfpatharc{-109}{-170}{4\myarrowsize}
	\pgfpatharc{10}{189}{0.58\pgflinewidth and 0.2\pgflinewidth}
	\pgfpatharc{-170}{-115}{4\myarrowsize+\pgflinewidth}
	\pgfpathclose
	\pgfusepathqfillstroke
	\pgfpathmoveto{\pgfpoint{1.5\pgflinewidth-0.3em}{0}}
	\pgfpatharc{109}{170}{4\myarrowsize}
	\pgfpatharc{-10}{-189}{0.58\pgflinewidth and 0.2\pgflinewidth}
	\pgfpatharc{170}{115}{4\myarrowsize+\pgflinewidth}
	\pgfpathclose
	\pgfusepathqfillstroke
	\pgfsetlinewidth{2\pgflinewidth}
}

\pgfarrowsdeclare{cmhook}{cmhook}{
	\pgfsetdash{}{0pt} 
	\pgfsetbeveljoin 
	\pgfsetroundcap 
	\setlength{\myarrowsize}{0.6pt}
	\addtolength{\myarrowsize}{.5\pgflinewidth}
	\pgfarrowsleftextend{-4\myarrowsize-.5\pgflinewidth} 
	\pgfarrowsrightextend{.8\pgflinewidth}
}{
	\setlength{\myarrowsize}{0.6pt} 
  	\addtolength{\myarrowsize}{.5\pgflinewidth}  
 	\pgfsetdash{}{0pt}
	\pgfsetroundcap
	\pgfpathmoveto{\pgfqpoint{0pt}{-4.667\pgflinewidth}}
	\pgfpathcurveto
    {\pgfqpoint{4\pgflinewidth}{-4.667\pgflinewidth}}
    {\pgfqpoint{4\pgflinewidth}{0pt}}
    {\pgfpointorigin}
	\pgfusepathqstroke
}

\newenvironment{diagram*}[2]{%
\[%
\begin{tikzpicture}[>=cmto,baseline=(current bounding box.center),%
	to/.style={->,font=\scriptsize,cap=round},%
	into/.style={cmhook->,font=\scriptsize,cap=round},%
	onto/.style={-cmonto,font=\scriptsize,cap=round},%
	math/.style={matrix of math nodes, row sep=#2, column sep=#1,%
		text height=1.5ex, text depth=0.25ex}]%
}{%
\end{tikzpicture}%
\]%
\ignorespacesafterend%
}

\newcommand{\Dmod}{\mathscr{D}}

\newcommand{\cohH}{\mathcal{H}}

\newcommand{\ZZ}{\mathbb{Z}}
\newcommand{\QQ}{\mathbb{Q}}

\newcommand{\CC}{\mathbb{C}}

\DeclareMathOperator{\HRH}{HRH}

\newcommand{\shf}[1]{\mathscr{#1}}

\def\overbar#1#2#3{{%
	\setbox0=\hbox{\displaystyle{#1}}%
	\dimen0=\wd0
	\advance\dimen0 by -#2 
	\vbox {\nointerlineskip \moveright #3 \vbox{\hrule height 0.3pt width \dimen0}%
		\nointerlineskip \vskip 1.5pt \box0}%
}}

\newcommand{\shO}{\shf{O}}

\newcommand{\cK}{\mathcal{K}}
\newcommand{\cL}{\mathcal{L}}
\newcommand{\cM}{\mathcal{M}}
\newcommand{\cN}{\mathcal{N}}

\newcommand{\de}{\partial}

\makeatletter
\let\@@seccntformat\@seccntformat
\renewcommand*{\@seccntformat}[1]{%
  \expandafter\ifx\csname @seccntformat@#1\endcsname\relax
    \expandafter\@@seccntformat
  \else
    \expandafter
      \csname @seccntformat@#1\expandafter\endcsname
  \fi
    {#1}%
}
\newcommand*{\@seccntformat@subsection}[1]{%
  \textbf{\csname the#1\endcsname.}
}
\makeatother

\makeatletter
\let\@paragraph\paragraph
\renewcommand*{\paragraph}[1]{%
	\vspace{0.3\baselineskip}%
	\@paragraph{\textit{#1}}%
}
\makeatother

\counterwithin{equation}{subsection}
\counterwithout{subsection}{section}
\counterwithin{figure}{subsection}

\newtheorem{theorem}[equation]{Theorem}
\newtheorem*{theorem*}{Theorem}
\newtheorem{lemma}[equation]{Lemma}
\newtheorem*{lemma*}{Lemma}
\newtheorem{corollary}[equation]{Corollary}
\newtheorem{proposition}[equation]{Proposition}
\newtheorem{exprop}[equation]{Example/Proposition}
\newtheorem*{proposition*}{Proposition}

\theoremstyle{definition}
\newtheorem{definition}[equation]{Definition}
\newtheorem*{definition*}{Definition}
\newtheorem{remark}[equation]{Remark}

\newtheorem{example}[equation]{Example}
\newtheorem*{example*}{Example}
\newtheorem*{problem*}{Problem}

\theoremstyle{plain}

\newcommand{\theoremref}[1]{\hyperref[#1]{Theorem~\ref*{#1}}}
\newcommand{\lemmaref}[1]{\hyperref[#1]{Lemma~\ref*{#1}}}
\newcommand{\definitionref}[1]{\hyperref[#1]{Definition~\ref*{#1}}}
\newcommand{\propositionref}[1]{\hyperref[#1]{Proposition~\ref*{#1}}}
\newcommand{\conjectureref}[1]{\hyperref[#1]{Conjecture~\ref*{#1}}}
\newcommand{\corollaryref}[1]{\hyperref[#1]{Corollary~\ref*{#1}}}
\newcommand{\exampleref}[1]{\hyperref[#1]{Example~\ref*{#1}}}
\newcommand{\setupref}[1]{\hyperref[#1]{Set-up~\ref*{#1}}}
\newcommand{\remarkref}[1]{\hyperref[#1]{Remark~\ref*{#1}}}
\newcommand{\claimref}[1]{\hyperref[#1]{Claim~\ref*{#1}}}
\newcommand{\figureref}[1]{\hyperref[#1]{Figure~\ref*{#1}}}

\makeatletter
\let\old@caption\caption
\renewcommand*{\caption}[1]{%
	\setcounter{figure}{\value{equation}}%
	\stepcounter{equation}%
	\old@caption{#1}\relax%
}
\makeatother

\newcounter{intro}

\newtheorem{intro-conjecture}[intro]{Conjecture}
\newtheorem{intro-corollary}[intro]{Corollary}
\newtheorem{intro-theorem}[intro]{Theorem}

\def\cK{\mathcal{K}}

\def\cA{\mathcal{A}}

\newcommand{\parref}[1]{\hyperref[#1]{\S\ref*{#1}}}

\makeatletter
\newcommand*\if@single[3]{%
  \setbox0\hbox{${\mathaccent"0362{#1}}^H$}%
  \setbox2\hbox{${\mathaccent"0362{\kern0pt#1}}^H$}%
  \ifdim\ht0=\ht2 #3\else #2\fi
  }
\newcommand*\rel@kern[1]{\kern#1\dimexpr\macc@kerna}
\newcommand*\widebar[1]{\@ifnextchar^{{\wide@bar{#1}{0}}}{\wide@bar{#1}{1}}}
\newcommand*\wide@bar[2]{\if@single{#1}{\wide@bar@{#1}{#2}{1}}{\wide@bar@{#1}{#2}{2}}}
\newcommand*\wide@bar@[3]{%
  \begingroup
  \def\mathaccent##1##2{%
    \if#32 \let\macc@nucleus\first@char \fi
    \setbox\z@\hbox{$\macc@style{\macc@nucleus}_{}$}%
    \setbox\tw@\hbox{$\macc@style{\macc@nucleus}{}_{}$}%
    \dimen@\wd\tw@
    \advance\dimen@-\wd\z@
    \divide\dimen@ 3
    \@tempdima\wd\tw@
    \advance\@tempdima-\scriptspace
    \divide\@tempdima 10
    \advance\dimen@-\@tempdima
    \ifdim\dimen@>\z@ \dimen@0pt\fi
    \rel@kern{0.6}\kern-\dimen@
    \if#31
      \overline{\rel@kern{-0.6}\kern\dimen@\macc@nucleus\rel@kern{0.4}\kern\dimen@}%
      \advance\dimen@0.4\dimexpr\macc@kerna
      \let\final@kern#2%
      \ifdim\dimen@<\z@ \let\final@kern1\fi
      \if\final@kern1 \kern-\dimen@\fi
    \else
      \overline{\rel@kern{-0.6}\kern\dimen@#1}%
    \fi
  }%
  \macc@depth\@ne
  \let\math@bgroup\@empty \let\math@egroup\macc@set@skewchar
  \mathsurround\z@ \frozen@everymath{\mathgroup\macc@group\relax}%
  \macc@set@skewchar\relax
  \let\mathaccentV\macc@nested@a
  \if#31
    \macc@nested@a\relax111{#1}%
  \else
    \def\gobble@till@marker##1\endmarker{}%
    \futurelet\first@char\gobble@till@marker#1\endmarker
    \ifcat\noexpand\first@char A\else
      \def\first@char{}%
    \fi
    \macc@nested@a\relax111{\first@char}%
  \fi
  \endgroup
}
\makeatother

\usepackage{xpatch}
\makeatletter   
\xpatchcmd{\@tocline}
{\hfil\hbox to\@pnumwidth{\@tocpagenum{#7}}\par}
{\ifnum#1<0\hfill\else\dotfill\fi\hbox to\@pnumwidth{\@tocpagenum{#7}}\par}
{}{}
\makeatother    

\makeatletter
\def\l@section{\@tocline{1}{0pt}{0pc}{0pc}{\bfseries}}
 \def\l@subsection{\@tocline{2}{0pt}{4pc}{6pc}{}}
\def\l@subsubsection{\@tocline{3}{0pt}{8pc}{8pc}{}}
 \makeatother

\begin{document}

\author[B.~Dirks]{Bradley Dirks}

\address{Department of Mathematics, Stony Brook University, Stony Brook, NY 11794-3651, USA}

\email{bradley.dirks@stonybrook.edu}

\author[S.~Olano]{Sebasti\'{a}n Olano}

\address{Department of Mathematics, University of Toronto, 40 St. George St., Toronto, Ontario Canada, M5S 2E4}

\email{seolano@math.toronto.edu}

\author[D.~Raychaudhury]{Debaditya Raychaudhury}

\address{Department of Mathematics and Statistics, University of New Mexico, Albuquerque, NM 87131, USA}

\email{rcdeba@gmail.com}

\thanks{This material is based upon work supported by the National Science Foundation under Grant No. DMS-1926686 and MSPRF DMS-2303070. The research of DR is partially supported by an AMS-Simons Travel Grant.}

\title[A Hodge Theoretic Generalization of $\mathbb{Q}$-Homology Manifolds: General Case]{A Hodge Theoretic Generalization of $\mathbb{Q}$-Homology Manifolds I: General Case}


\begin{abstract} We study a natural Hodge-theoretic generalization of rational (or $\mathbb{Q}$-)homology manifolds through an invariant $\HRH(Z)$ attached to a complex algebraic variety $Z$. The defining property of this notion encodes the difference between higher Du Bois and higher rational singularities for local complete intersections, which are two classes of singularities that have recently gained much attention.
We show that $\HRH(Z)$ can be characterized when the variety $Z$ is embedded into a smooth variety using the local cohomology mixed Hodge modules. Near a point, this is also characterized by the local cohomology of $Z$ at the point, and hence, by the cohomology of the link. We give an application to partial Poincar\'{e} duality. We also introduce the generic local cohomological defect ${\rm lcdef}_{\textrm{gen}}(Z)$ and relate it to $\HRH(Z)$. Various examples are discussed at the end.

\end{abstract}


\maketitle




\section{Introduction}
The cohomology of a compact, oriented real manifold $X$ of dimension $2n$ satisfies the following remarkable symmetry, called Poincar\'{e} duality:
\[ H^{n-k}(X,\QQ) \cong H^{n+k}(X,\QQ)^{\vee},\]
where the right-hand side is the dual vector space.

In the non-compact setting, such a duality still holds, but it compares singular cohomology with compactly supported cohomology:
\[ H^{n-k}(X,\QQ) \cong H^{n+k}_{\rm c}(X,\QQ)^{\vee}.\]

The above applies to smooth complex algebraic varieties of dimension $n$. For singular varieties (in this paper, meaning reduced, finite type schemes over $\CC$), the duality need not hold for singular cohomology of the associated analytic space. Goresky-MacPherson's theory of intersection cohomology provides a replacement of singular cohomology which still admits a Poincar\'{e} duality isomorphism. This theory associates to any purely $d$-dimensional complex variety $Z$ a pair of graded vector spaces ${\rm IH}^*(Z,\QQ)$ and ${\rm IH}^*_{\rm c}(Z,\QQ)$ (which, in the smooth case, agree with the singular cohomology and compactly supported cohomology, respectively), such that there are natural isomorphisms
\[ {\rm IH}^{d-k}(Z,\QQ) \cong {\rm IH}_{c}^{d+k}(Z,\QQ)^{\vee}.\]

As in \cite{BBDG}, these vector spaces can be realized as the hypercohomology of the intersection complex perverse sheaf, which is self-dual as a perverse sheaf. The purpose of this article is to introduce an invariant of complex algebraic varieties measuring a partial Poincar\'e duality property.

One of the many achievements of Saito's theory of mixed Hodge modules \cites{SaitoMHP,SaitoMHM} is that it endows these intersection cohomology spaces with natural mixed Hodge structures. We review the aspects of this theory used throughout the paper in Section \ref{sect-Preliminaries}.

Let $Z$ be a purely $d$-dimensional complex algebraic variety. Recall that if $Z$ is smooth, there is a perfect pairing \[ \Omega^p_Z \times \Omega^{d-p}_Z \to \omega_Z, \] and therefore isomorphisms  \[\Omega^p_Z \stackrel{\cong}{\to} (\Omega^{d-p}_Z)^*\otimes_{\shO_Z} \omega_Z = \mathbb D_Z(\Omega^{d-p}_Z)[-d],\]
where $\mathbb D_Z = R \cohH om_{\shO_Z}(-,\omega_Z^\bullet)$ is the Grothendieck duality functor. In general, this picture is partially generalized by morphisms in $D^b_{\rm coh}(\shO_Z)$:   
\[ \phi^p \colon \underline{\Omega}_Z^p \to \mathbb D_Z(\underline{\Omega}_Z^{d-p})[-d],\]
where $\underline{\Omega}_Z^p$ is the $p$-th Du Bois complex of $Z$. We review the Du Bois complexes in Section \ref{sect-Preliminaries} below. For now, these should be thought of as replacements of the sheaf of K\"{a}hler differentials which are better behaved from a Hodge-theoretic point of view. 

\begin{definition}\label{deff} Let $Z$ be a pure $d$-dimensional variety. We say $Z$ is a \emph{rational homology manifold to Hodge degree} $k$, or {\it a $k$-Hodge rational homology variety}  if $\phi^p$ is a quasi-isomorphism for $0\leq p\leq k$. 

Further, define the {\it HRH level of $Z$} to be
\[{\rm HRH}(Z) := \sup\{k \in \ZZ_{\geq -1} \mid \phi^p \text{ is a quasi-isomorphism for } 0\leq p\leq k\}\]
where we follow the convention that $\HRH(Z) = -1$
if $\phi^0$ is not a quasi-isomorphism.  
\end{definition}

\begin{remark} The interest in this invariant arises naturally due to its connection to \emph{higher singularities} of the variety $Z$. Indeed, it is a consequence of the higher injectivity theorem (for isolated singularities due to Popa-Shen-Vo \cite{PSV} and for the arbitrary case due to Kov\'{a}cs \cite{KovacsInjectivity} and Chen-Dirks-Olano \cite{CDOInjectivity}) that we have the equivalence (for various notions of higher Du Bois and rational singularities):
\begin{equation} \label{eq-HigherSings} Z \text{ has } m\text{-rational singularities} \iff Z \text{ has } m\text{-Du Bois singularities and } {\rm HRH}(Z) \geq m.\end{equation}

The implication from left to right as in \eqref{eq-HigherSings} has a long history, dating to the conjecture of Steenbrink and its resolution, a theorem of Kov\'{a}cs \cite{KovacsDBRat}*{Thm. S} and independently Saito \cite{SaitoOnHodgeFilt}: rational singularities are Du Bois.
\end{remark}

\noindent{\bf Applications of HRH level.} As mentioned, we have partial Poincar\'e duality for $k$-Hodge rational homology varieties (see \theoremref{thm-PD} for a more elaborate formulation). Throughout the article, a variety $Z$ is \emph{embeddable} if there exists a closed embedding $i\colon Z \to X$ with $X$ a smooth variety.

\begin{intro-theorem}[Partial Poincar\'e duality]\label{thm-pd} Let $Z$ be an embeddable complex algebraic variety. If $\HRH(Z) \geq k$, then for all $i\in \ZZ$, we have isomorphisms for all $p \leq k$:
\[ {\rm Gr}_F^{d-p} H^{d-i}(Z) \cong ({\rm Gr}_F^{p} H^{d+i}_c(Z))^{\vee}.\]
{If $Z$ is proper with isolated non-rational homology manifold locus, then the converse holds.}
\end{intro-theorem}

The latter statement also gives our first characterization of ${\rm HRH}(Z)$. Below we give several others in terms of well-known characterizations of the rational homology manifold condition.

\begin{remark} This notion is different from another weakening of Poincar\'{e} duality, due to Kato \cite{Kato}. Indeed, the notion we study is related to which Hodge filtered pieces are Poincar\'{e} dual to each other (in all cohomological degrees), whereas Kato's notion asks for which cohomological degrees Poincar\'{e} duality holds (in all Hodge levels).
\end{remark}

The notion of $k$-Hodge rational homology varieties is also studied in the recent article \cite{PPLefschetz} through an equivalent defining property that is phrased as the condition $D_k$ which requires the natural maps in $D^b_{\rm coh}(\shO_Z)$
\[
\underline{\Omega}_Z^p\to {\rm I} \underline{\Omega}_Z^p
\]
to be isomorphisms for $p\leq k$, where ${\rm I} \underline{\Omega}_Z^p$ is the $p$-th intersection Du Bois complex (see \remarkref{rmk-compareStark} for the equivalence of the $D_k$-condition with HRH$(Z)\geq k$). This equivalence immediately gives an interesting concrete application of the invariant ${\rm HRH}(Z)$ via a dual Kodaira-Akizuki-Nakano type vanishing: 


\begin{intro-corollary}[KAN type vanishing]\label{cor-DANVanishing} Let $Z$ be a purely $d$-dimensional projective variety with an ample line bundle $L$. Then we have vanishing
\[ \mathbb H^q(Z, \underline{\Omega}_Z^p \otimes_{\shO_Z} L^{-1}) = 0 \text{ for all } p+q < d, \, p \leq {\rm HRH}(Z).\]
\end{intro-corollary}

Another way to see this result is by applying \cite{CDOCCI}*{Thm. C}.

In \cite{PPLefschetz}, important consequences of condition $D_k$, such as symmetry of Hodge-Du Bois numbers and Lefschetz properties, are discussed in detail. 
Moreover, in \cite{PPLefschetz}*{Cor. 7.5}, it was proved that the codimension of the locus $Z_{\textrm{nRS}}$ where $Z$ is not a rational homology manifold is bounded below by $2\HRH(Z)+3$ (here nRS stands for non-rationally smooth). 

In fact, they established an inequality involving ${\rm HRH}(Z)$ and the \emph{local cohomological defect} of $Z$. This latter invariant is defined through local cohomology modules $\cohH^{j}_Z(\shO_X)$ when $Z\subseteq X$ is an embedding with $X$ smooth. These modules are potentially non-zero only for $j \geq {\rm codim}_X(Z)$, and we have
\[{\rm lcdef}(Z) := \max\{j\mid \cohH^{c+j}_Z(\shO_X)\neq 0, \text{ where } Z \subseteq X \text{ is a codimension }c\text{ embedding}\}\] 
It turns out that the above description does not depend on choice of the embedding. 

We introduce a ``generic variant'' of the invariant ${\rm lcdef}(Z)$ in \S \ref{sec-GenLcdef} that we call ${\rm lcdef}_{\textrm{gen}}(Z)$ (see \definitionref{deflcdefgen}). It is a non-negative integer satisfying the inequality 
\begin{equation}\label{obvneq}
    {\rm lcdef}_{\textrm{gen}}(Z)\leq {\rm lcdef}(Z)
\end{equation}
(equality holds in the case of isolated singularities, but strict inequality is also possible in the above, see \S \ref{sec-determinantal} for explicit examples).
We obtain the following improvement of \cite{PPLefschetz}*{Cor. 7.5} via ${\rm lcdef}_{\textrm{gen}}(Z)$:

\begin{intro-theorem}\label{thm-ppbound}
Let $Z$ be a purely $d$-dimensional variety with $\HRH(Z) \geq 0$. Then we have the inequality
\[ {\rm lcdef}_{\rm gen}(Z) + 2\HRH(Z) +3 \leq {\rm codim}_Z(Z_{\rm nRS}).\]
\end{intro-theorem}

A more elaborate version of the above is proven in \propositionref{prop-ppbound}. 
There are instances when equality holds, but strict inequality also occurs. These examples are in fact local complete intersections, which is the topic of the sequel \cites{DOR2}, so we refrain from discussing them here.

Various other geometric and topologically important properties of ${\rm HRH}(Z)$ are highlighted throughout this article. Because of this, it is important to develop techniques to measure this invariant, which is the main task that we undertake in this paper. 

\medskip

\noindent{\bf Characterizations of HRH level.} There are various well-known equivalent descriptions of the rational homology manifold property of a complex algebraic variety $Z$:
\begin{enumerate} \item If $Z\subseteq X$ is a closed subvariety of a smooth variety $X$ with $\dim X - \dim Z = c$, and if $\cohH^j_Z(\shO_X)$ are the \emph{local cohomology modules} of $\shO_X$ along $Z$, then 
\[ \cL(X,Z) = \cohH^c_Z(\shO_X) \text{ and } \cohH^q_Z(\shO_X) = 0\text{ for all } q > c,\]
where $\cL(X,Z)$ is the intersection homology $\Dmod$-module \cite{BrylinskiKashiwara}.

\item For all $x\in Z$, we have
\[ H^i_{\{x\}}(Z) = \begin{cases} 0 & i < 2\dim Z \\ \QQ & i = 2\dim Z \end{cases}.\]

\item If $\mathcal S = \{S_\alpha\}_{\alpha \in I}$ is a Whitney stratification of $Z$ and $L_\alpha$ is the link of $Z$ at $S_\alpha$, then
\[ H^i(L_\alpha) = \begin{cases} 0 & i < 2 {\rm codim}_Z(S_\alpha) \\ \QQ & i = 2 {\rm codim}_Z(S_\alpha)\end{cases}. \]
\end{enumerate}

The following theorem shows that the invariant ${\rm HRH}(Z)$ is a Hodge-theoretic weakening of \emph{any} of these characterizations of the rational homology manifold condition. We remark that the local cohomology modules admit a mixed Hodge module structure (see \cite{MustataPopaDuBois} for much more about this structure). In this setting, the intersection homology $\Dmod$-module underlies the lowest weighted piece of $\cohH^c_Z(\shO_X)$. In this paper, we index the Hodge filtration following conventions for \emph{right} $\Dmod$-modules.

The local cohomology vector space $H^i_{\{x\}}(Z)$ admits a mixed Hodge structure, and so does $H^i(L_\alpha)$, though $L_\alpha$ is in general not algebraic. We thank Lauren\c{t}iu Maxim for pointing out the third characterization in the following theorem.

\begin{intro-theorem} \label{thm-characterize} Let $Z$ be a purely $d$-dimensional variety. Then for any $k\in \ZZ_{\geq -1}$, the condition ${\rm HRH}(Z) \geq k$ is equivalent to any of the following equivalent conditions:
\begin{enumerate} \item \label{thmmain} If $Z\subseteq X$ is a closed subvariety of a smooth variety $X$ with $\dim X = n$ and $n - d = c$ then 
\[ F_{k-n} W_{n+c} \cohH_Z^c(\shO_X) = F_{k-n} \cohH^c_Z(\shO_X)\,\,\textrm{ and }\,\, F_{k-n} \cohH^q_Z(\shO_X) = 0 \text{ for all } q > c.\]

\item \label{propthm-LocCohAtPoint} For all $x\in Z$, we have
\[ F^{d-k} H^i_{\{x\}}(Z) = \begin{cases} 0 & i < 2d \\ \CC & i = 2d \end{cases}.\]

\item \label{thm-link} If $\mathcal S = \{S_\alpha\}_{\alpha \in I}$ is a Whitney stratification of $Z$ and $L_\alpha$ is the link of $Z$ at $S_\alpha$, then
\[ F^{{\rm codim}_Z(S_\alpha)-k} H^i(L_\alpha) = \begin{cases} 0 & i < 2 {\rm codim}_Z(S_\alpha) \\ \CC & i = 2 {\rm codim}_Z(S_\alpha)\end{cases}. \]
\end{enumerate}
\end{intro-theorem}

\begin{remark} As is convention, for a mixed Hodge structure $(V,F,W)$, we use Hodge filtrations that are indexed decreasingly. The corresponding increasing filtration is defined by $F_\bullet V = F^{-\bullet}V$.
\end{remark}

Combining \theoremref{thm-characterize}\eqref{thmmain} with a result of Musta\c{t}\u{a}-Popa \cite{MustataPopaDuBois}*{Thm. C}, we get the following. In the statement, the ``Ext'' filtration $E_\bullet \cohH^q_Z(\shO_X)$ is defined using the Ext description of local cohomology:
\[ \cohH^q_Z(\shO_X)= \varinjlim_{p}\mathcal E xt^q(\shO_X/\mathcal{J}_Z^{p+1},\shO_X),\]
where $\mathcal{J}_Z$ is the ideal sheaf of $Z$ in $X$, and the filtration is defined by
\[ E_\bullet \cohH^q_Z(\shO_X) = {\rm Im}\left[\mathcal E xt^q(\shO_X/\mathcal{J}_Z^{\bullet+1},\shO_X)\to \cohH^q_Z(\shO_X)\right].\]

\begin{intro-corollary} \label{cor-rational} If $Z \subseteq X$ is a closed embedding of a pure $c$-codimensional Cohen-Macaulay variety $Z$ inside a smooth variety $X$ of dimension $n$, then
\[ Z\text{ has rational singularities if and only if } F_{-n}W_{n+c}\cohH^c_Z(\shO_X) = E_0 \cohH^c_Z(\shO_X).\]
\end{intro-corollary}

An application of the characterization in terms of local cohomology \theoremref{thm-characterize}\eqref{thmmain} and \theoremref{thm-ppbound} is the computation of ${\rm HRH}(Z)$ in the case of determinantal varieties.

\begin{intro-corollary}\label{cor-det} Let $p \in \ZZ_{\geq 1}$ and let $Z_p$ denote a determinantal variety in the following four settings. 
\begin{enumerate} \item (Generic) $\HRH(Z_p) = 0$.

\item (Odd skew) $\HRH(Z_p) \in \{0,1\}$.

\item (Even skew) $\HRH(Z_p) \in \{0,1\}$.

\item (Symmetric) $\HRH(Z_p) \in \{0,1\}$ (when $p\geq 2$).
\end{enumerate}
    
\end{intro-corollary}

For more details on the definition of $Z_p$ see \S \ref{sec-determinantal}. Other examples are discussed in depth in Section \ref{sect-Examples} and in the companion paper \cite{DOR2}.

In the isolated singularities case, we can give another characterization of the invariant ${\rm HRH}(Z)$ via the \emph{link invariants}, following the notation of \cite{FLIsolated}. The statement of the corollary concerns the local variant of ${\rm lcdef}(Z)$ and ${\rm HRH}(Z)$, both defined by taking small enough Zariski open neighborhoods of the point $x\in Z$ and denoted with a subscript $x$.

 \begin{intro-corollary}\label{corlinkiso} Let $Z$ be a purely $d$-dimensional variety and let $x\in Z$ be an isolated singular point. Let $a = {\rm lcdef}_x(Z)$.
Then for any $k\geq 0$, we have $\HRH_x(Z) \geq k$ if and only if \begin{equation} \label{eq-LinkVanishing}  \ell^{d-i,q-d+i} = 0\text{ for all }  i\leq k, \, q\in [d,d+a]\end{equation}

Thus, $Z$ has $k$-rational singularities near $x$ if and only if it has $k$-Du Bois singularities near $x$ and the vanishing \eqref{eq-LinkVanishing} holds.
\end{intro-corollary}

We also recall the result of Brion \cite{Brion}*{Prop. A1} which states that if $Z$ is a rational homology manifold near $x$, then it is irreducible near $x$. In fact, using a bound on the local cohomological defect due to \cite{PPLefschetz}, we conclude the following (compare with the fact that if $Z$ has rational singularities near $x$, then it is normal, hence irreducible, near $x$):

\begin{intro-theorem}\label{thm-TopLocCoh} Let $Z$ be a purely $d$-dimensional variety and let $x\in Z$. If $\HRH_x(Z) \geq 0$, then
\[ \dim H^{2d-\ell-1}(L_x) = \dim H^{2d-\ell}_{\{x\}}(Z) = \begin{cases} 1 & \ell = 0 \\ 0 & 0 < \ell \leq 2{\rm HRH}_x(Z) +1 \end{cases}.\]
In particular, if ${\rm HRH}_x(Z) \geq 0$, then $Z$ is irreducible at $x$.
\end{intro-theorem}




\noindent\textbf{Outline.} Section \ref{sect-Preliminaries} contains a review of the theory of mixed Hodge modules, and the definitions of higher Du Bois and higher rational singularities.
Section \ref{sect-Definition} defines and studies the invariant $\HRH(Z)$. The proofs of \corollaryref{cor-DANVanishing}, \theoremref{thm-characterize}\eqref{thmmain} and \corollaryref{cor-rational} 
are given in \S \ref{emb-sec}. In the following \S \ref{link-sec}, \theoremref{thm-characterize}\eqref{propthm-LocCohAtPoint} ($=$ \theoremref{prop-LocCohAtPoint}), \theoremref{thm-characterize}\eqref{thm-link}, \corollaryref{corlinkiso} and \theoremref{thm-TopLocCoh} are proven. Moreover, an application to partial Poincar\'{e} duality on singular cohomology is given in \S \ref{sec-pd} where we prove \theoremref{thm-pd} ($=$ \theoremref{thm-PD}).
The following \S \ref{sec-GenLcdef} contains a proof of \theoremref{thm-ppbound} ($=$ \propositionref{prop-ppbound}), the main observation being that ${\rm lcdef}_x(Z)$ is invariant under taking normal slices. The end of that section highlights some interesting behavior with references to the examples in \S \ref{sec-determinantal}.
Section \ref{sect-Examples} provides examples of various features. Here we compute the Hodge rational homology levels of affine cones, determinantal varieties and other natural classes of examples. \corollaryref{cor-det} ($=$ \propositionref{prop-ComputeHRHDet}) is proven in this section.

\medskip

\noindent {\bf Acknowledgments.} We would like to thank Bhargav Bhatt, Qianyu Chen, Radu Laza, Lauren\c{t}iu Maxim, Mircea Musta\c{t}\u{a}, Sung Gi Park, Mike Perlman, Mihnea Popa, Sridhar Venkatesh and Anh Duc Vo for many conversations on the topics in this paper.

\section{Preliminaries} \label{sect-Preliminaries}
In this section, we give a brief overview of the background material needed in the rest of the paper. We will use without review the theory of perverse sheaves and $\Dmod$-modules. For more information, see \cite{BBDG} and \cite{HTT}, respectively.

\subsection{Mixed Hodge modules} The main objects used in this paper are mixed Hodge modules, defined by Saito \cites{SaitoMHP,SaitoMHM}. We make the convention in this paper that all $\Dmod$-modules are \emph{left} modules; however, we will index the Hodge filtration following the conventions for right $\Dmod$-modules. We will remind the reader about these conventions below, when necessary.

On a smooth complex algebraic variety $X$ of dimension $n$, a \emph{mixed Hodge module} consists of the data 
\[M = (\cM,F,W,(\cK,W),\alpha)\]
where $\cM$ is a regular holonomic $\Dmod_X$-module, $F_\bullet \cM$ is a good filtration on it, $W_\bullet \cM$ is a finite filtration by $\Dmod_X$-modules, $(\cK,W)$ is an algebraically constructible $\QQ$-perverse sheaf on $X^{\rm an}$ with a finite filtration $W_\bullet \cK$, and $\alpha$ is a filtered isomorphism
\[ \alpha \colon \CC \otimes_{\QQ} (\cK,W) \to {\rm DR}^{\rm an}_X(\cM,W)\]
of filtered $\CC$-perverse sheaves. Recall that
\[ {\rm DR}_X(\cM) = \left[\cM \xrightarrow[]{\nabla} \Omega_X^1\otimes_{\shO}\cM \xrightarrow[]{\nabla}\dots \xrightarrow[]{\nabla} \omega_X \otimes_{\shO} \cM\right]\]
placed in degrees $-n,\dots, 0$, with filtration $W_i {\rm DR}_X(\cM) = {\rm DR}_X(W_i \cM)$. In a local choice of coordinates $x_1,\dots, x_n$ of $X$, the complex is the Koszul complex on the operators $\de_{x_1},\dots, \de_{x_n}$. 

The filtration $F_\bullet \cM$ is the ``Hodge filtration'' and $W_\bullet \cM$ is the ``weight filtration''. The $\QQ$-perverse sheaf $\cK$ is called the \emph{$\QQ$-structure}, and the functor ${\rm rat}\colon {\rm MHM}(X) \to {\rm Perv}(X)$ sending $M$ to $\cK$ is faithful.

These data are subject to various conditions, which we will not explain fully here. The essential idea is that mixed Hodge modules on a point should be exactly the graded-polarizable mixed Hodge structures, and then the definition for higher dimension varieties follows by induction on the dimension.

The category ${\rm MHM}(X)$ is abelian. In fact, any morphism $\varphi \colon (\cM,F,W) \to (\cN,F,W)$ underlying a morphism of mixed Hodge modules is bi-strict with respect to $F$ and $W$. We say a mixed Hodge module $M$ is \emph{pure of weight} $w$ if ${\rm Gr}^W_i \cM = 0$ for all $i\neq w$.

We let $D^b({\rm MHM}(X))$ denote the bounded derived category of mixed Hodge modules on $X$. 

The theory of mixed Hodge modules is endowed with a six functor formalism which, under the functor ${\rm rat}\colon D^b({\rm MHM}(X)) \to D^b({\rm Perv}(X))$ agrees with the six functor formalism on perverse sheaves and which agrees with the six functors on underlying $\Dmod$-modules. In particular, given any morphism $f\colon X \to Y$ between smooth varieties, we have functors
\[ f_*,\, f_! \colon D^b({\rm MHM}(X)) \to D^b({\rm MHM}(Y)),\]
\[ f^*,f^! \colon D^b({\rm MHM}(Y)) \to D^b({\rm MHM}(X)),\]
with $f^*$ left adjoint to $f_*$, $f_!$ left adjoint to $f^!$. Moreover, there is an exact functor
\[ \mathbf D_X \colon {\rm MHM}(X)^{\rm op} \to {\rm MHM}(X)\]
so that $f^! = \mathbf D_X f^* \mathbf D_Y, f_! = \mathbf D_Y f_* \mathbf D_X$. The dual functor satisfies
\[ {\rm Gr}^W_{-i} \mathbf D_X(M) \cong \mathbf D_X({\rm Gr}^W_i M).\]

Using local embeddings into smooth varieties, the categories ${\rm MHM}(Z)$ and $D^b({\rm MHM}(Z))$ make sense for an arbitrary complex variety $Z$, and admit six functor formalisms as described above. Similarly, the associated graded pieces ${\rm Gr}^F_p {\rm DR}_Z(M)$ give objects of $D^b_{\rm coh}(\shO_Z)$ which are independent of the choice of local smooth embeddings.

For any two smooth varieties $X,Y$ and for any $M\in {\rm MHM}(X)$, the functor
\[ M \boxtimes - \colon {\rm MHM}(Y) \to {\rm MHM}(X\times Y)\]
is exact. On underlying filtered objects, it is given by convolution of filtrations: we have
\[ F_k (\cM\boxtimes \cN) = \sum_{i+j = k} F_i \cM \boxtimes F_j \cN,\]
\[ W_k (\cM\boxtimes \cN) = \sum_{i+j= k}W_i \cM \boxtimes W_j \cN.\]

\begin{example} For $X$ a smooth variety of dimension $n$, the \emph{trivial Hodge module} is
\[ \QQ_X^H[n] = (\shO_X, F,W,\underline{\QQ}_{X^{\rm an}}[n]),\]
where ${\rm Gr}^F_{-\bullet} \shO_X = {\rm Gr}^W_\bullet \shO_X = 0$ except for $\bullet = n$.

In general, given the trivial Hodge structure $\QQ^H \in {\rm MHM}({\rm pt})$, if $a_Z\colon Z \to {\rm pt}$ is the constant map, then the trivial Hodge module on $Z$ is actually an object in $D^b({\rm MHM}(Z))$ given by
\[ \QQ^H_Z = a_Z^* \QQ^H,\]
which might have many non-zero cohomology modules and those modules may not be pure.
\end{example}

\begin{example}(\cite{SaitoMHM}*{(4.4.2)}) \label{eg-SmoothPullback} Assume $X$ and $Y$ are smooth varieties. Let $p\colon X\times Y \to Y$ be the projection. Then the pullback functor $p^*$ is given by $\QQ^H_X \boxtimes -$.
\end{example}

\begin{example} (\cite{SaitoMHM}*{(4.4.3)}) \label{eg-BaseChange} Consider a Cartesian diagram
\[ \begin{tikzcd} Y' \ar[d,"g_Y"] \ar[r,"f'"] & X' \ar[d,"g_X"] \\ Y \ar[r,"f"] & X \end{tikzcd}.\]

Then there are natural, canonical isomorphisms of functors
\[ g_X^* f_! = f'_! g_Y^*,\quad g_X^! f_* = f'_* g_Y^!.\]
\end{example}

\begin{example} \label{eg-TateTwist} For any $M \in {\rm MHM}(X)$ and $j\in \ZZ$, we can define another mixed Hodge module $M(j) \in {\rm MHM}(X)$, which is called the \emph{Tate twist of $M$ by $j$}. It has the same underlying $\Dmod_X$-module $\cM$, but the filtrations are shifted:
\[ F_{\bullet}(\cM(j)) = F_{\bullet-j}(\cM),\, W_\bullet(\cM(j)) = W_{\bullet+2j} \cM.\]
\end{example}

If $M$ is a pure Hodge module of weight $w$ on $X$, then by definition it is \emph{polarizable}, which implies that there exists an isomorphism of pure Hodge modules of weight $-w$:
\[ \mathbf D_X(M) \cong M(w).\]

The Hodge filtration induces a filtration on ${\rm DR}_X(\cM)$ by
\[ F_p {\rm DR}_X(\cM) = \left[ F_{p-n} \cM \xrightarrow[]{\nabla} \Omega_X^1 \otimes_{\shO} F_{p-n+1}\cM \xrightarrow[]{\nabla} \dots \xrightarrow[]{\nabla} \omega_X \otimes_{\shO} F_p \cM \right],\]
so that ${\rm Gr}^F_p {\rm DR}_X(\cM)$ is actually a bounded complex of coherent $\shO$-modules with $\shO$-linear differentials. The functor ${\rm Gr}^F_p {\rm DR}_X(-)$ extends to an exact functor
\[ {\rm Gr}^F_p {\rm DR}_X(-) \colon D^b({\rm MHM}(X)) \to D^b_{\rm coh}(\shO_X),\]
where the right-hand side is the bounded derived category of $\shO_X$-modules with coherent cohomology. 

Moreover, the functor is well-behaved under various operations, as given by the following proposition:
\begin{proposition}[\cite{SaitoMHP}*{Lem. 2.3.6}] \label{prop-GrDRProps} Let $f\colon X \to Y$ be a proper morphism between smooth varieties and let $M^\bullet \in D^b({\rm MHM}(X))$. Then, for any $p\in \ZZ$ there is a quasi-isomorphism
\[ {\rm Gr}^F_p {\rm DR}_Y(f_*(M^\bullet)) \cong  Rf_* {\rm Gr}^F_p {\rm DR}_X(M^\bullet),\]
where $Rf_*$ is the right derived functor of the usual $\shO$-module push-forward.

Moreover, by \cite{SaitoMHP}*{Sect. 2.4}, we have
\[\mathbb D_X {\rm Gr}^F_p {\rm DR}_X(M^\bullet) \cong {\rm Gr}^F_{-p} {\rm DR}_X(\mathbf D_X(M^\bullet)),\]
where $\mathbb D_X(-) = R \mathcal Hom_{\shO_X}(-,\omega_X[\dim X])$ is the Grothendieck duality functor on $X$.
\end{proposition}

Given any object $A$ with bounded below filtration $F_\bullet A$, we let $p(A,F) = \min\{p \mid F_p A \neq 0\}$. For $M$ a mixed Hodge module, we let $p(M) = p(\cM,F)$ where $F_\bullet \cM$ is the Hodge filtration on the underlying $\Dmod$-module. For $M^\bullet \in D^b({\rm MHM}(X))$, we have $p(M^\bullet) = \min_{i\in \ZZ} p(\cohH^i(M^\bullet))$.


\begin{lemma} \label{lem-technical} Let $M^\bullet \in D^b({\rm MHM}(X))$. Then
\[ p(M^\bullet) = \min\{p \mid {\rm Gr}^F_p {\rm DR}_X(M^\bullet) \text{ is not acyclic}.\}\]

\end{lemma}
\begin{proof} The argument is standard, but we include it for convenience of the reader.

We want to see that
\[ F_k M^\bullet = 0 \text{ if and only if } {\rm Gr}^F_\ell {\rm DR}_X(M^\bullet) = 0 \text{ for all } \ell \leq k,\]
where the first expression is equivalent to saying $p(M^\bullet) > k$. Indeed, if we represent $M^\bullet$ by a bounded complex of morphisms of mixed Hodge modules, then
\[ F_k \cohH^j(\cM^\bullet) = \cohH^j(F_k \cM^\bullet)\]
by strictness of morphisms.

We have the spectral sequence (shown using the standard truncation functors on $D^b({\rm MHM}(X))$):
\[ {}^{\ell} E_2^{p,q} = \cohH^p {\rm Gr}_\ell^F {\rm DR}_X(\cohH^q M^\bullet) \implies \cohH^{p+q} {\rm Gr}^F_\ell {\rm DR}_X(M^\bullet).\]

Note that if $F_k M^\bullet =0$ (meaning $F_k \cohH^j M^\bullet = 0$ for all $j\in \ZZ$), then ${}^\ell E_2^{p,q} = 0$ for all $\ell \leq k$ and all $p,q$. The spectral sequence then shows that ${\rm Gr}^F_\ell {\rm DR}_X(M^\bullet) = 0$ for all $\ell \leq k$, as desired.

Conversely, assume ${\rm Gr}^F_\ell {\rm DR}_X(M^\bullet) = 0$ for all $\ell \leq k$. Let $\sigma_0 = \min\{\sigma \mid F_\sigma M^\bullet \neq 0\}$, which is a finite value because $M^\bullet$ is a bounded complex (meaning there are only finitely many cohomology modules to consider). The claim is that $\sigma_0 > k$. If not, then $\sigma_0 \leq k$, and so by our assumed vanishing, we have
\[ {}^{\sigma_0} E_\infty^{p,q} = 0.\]

But for any fixed $q$, the only non-zero ${}^{\sigma_0} E_2^{p,q}$ is for $p = 0$. Indeed, the last few terms of the associated graded de Rham complex are
\[ \dots \to \Omega_X^{n-1} \otimes {\rm Gr}^F_{\sigma_0 -1} \cohH^q(M^\bullet) \xrightarrow[]{d} \omega_X \otimes {\rm Gr}^F_{\sigma_0} \cohH^q(M^\bullet),\]
and by definition of $\sigma_0$, all the leftmost terms are 0. Thus, ${}^{\sigma_0} E_\infty^{p,q} = {}^{\sigma_0} E_2^{p,q} = 0$.

So we have reduced to checking the claim when $M$ is a mixed Hodge module. But it is easy to see that $F_k M = 0$ if and only if ${\rm Gr}^F_{\ell} {\rm DR}_X M = 0$ for all $\ell \leq k$.
\end{proof}

\begin{corollary}\label{cor-technical}
 Let $\psi \colon M^\bullet \to N^\bullet$ be a morphism in $D^b({\rm MHM}(X))$. Then $F_k \psi$ is a quasi-isomorphism if and only if ${\rm Gr}^F_\ell {\rm DR}_X(\psi)$ is a quasi-isomorphism for all $\ell \leq k$.
\end{corollary}
\begin{proof} Apply \lemmaref{lem-technical} to the cone of the morphism $\psi$ in $D^b({\rm MHM}(X))$.
\end{proof}

The following two lemmas are proven in a way similar to \cite{SaitoMHM}*{Rmk. 4.6(1)}. The main idea is to establish the result for variations of mixed Hodge structures and to use induction on the dimension of the support.

\begin{lemma} \label{lem-pInvariantVanishing} Let $M^\bullet \in D^b({\rm MHM}(Z))$. Then
\[ p(M^\bullet) \geq j \iff p(i_x^! M^\bullet) \geq j \text{ for all } x\in Z.\]
\end{lemma}
\begin{proof} The implication 
\[ p(M^\bullet) \geq j \implies p(i_x^! M^\bullet) \geq j \text{ for all } x\in Z\]
is obvious, by definition of the functor $i_x^!$ for mixed Hodge modules (see, for example \cite{CD}*{Thm. B}).

To prove the converse, assume for all $x$ that $p(i_x^! M^\bullet) \geq j$.

We use induction on $\dim(Z)$. For $\dim(Z) = 0$, the claim is obvious.

For $\dim(Z) > 0$, there exists a Zariski open cover $Z = \bigcup_{\alpha \in I} U_\alpha$ such that, for each $\alpha \in I$, there exists $g_\alpha \in \shO(U_\alpha)$ so that the subset $U'_\alpha = \{g_\alpha \neq 0\} \subseteq U_\alpha$ is a smooth, dense open subset such that the restriction $\cohH^i(M^\bullet)\vert_{U'_\alpha}$ is a variation of mixed Hodge structures for all $i \in \ZZ$.

It suffices to prove the claim locally, so we can replace $Z$ with $U_\alpha$. We have reduced to the case that there exists $g\in \shO_Z(Z)$ such that $U' = \{g\neq 0\} \subseteq Z$ is a smooth, Zariski open dense subset so that $\cohH^i(M^\bullet)\vert_{U'}$ is a variation of mixed Hodge structures for all $i\in \ZZ$.

Let $j\colon U' \to Z$ and $i \colon \{g=0\} \to Z$ be the natural embeddings, with exact triangle
\[ i_* i^! M^\bullet \to M^\bullet \to j_*(M^\bullet\vert_{U'}) \xrightarrow[]{+1}.\]

If $p(M^\bullet) < j$, then either $p(i_* i^! M^\bullet) < j$ or $p(j_* (M^\bullet\vert_{U'})) < j$. Note that for all $x\in \{g=0\}$, we have
\[ i_x^! i_* i^! M^\bullet = i_x^! M^\bullet,\]
and so by induction on the dimension we conclude that $p(i_*i^!M^\bullet) \geq j$.

For all $x\in U'$, we have
\[ i_x^! M^\bullet = i_x^! j_*(M^\bullet) = \iota_x^!(M^\bullet\vert_{U'}),\]
where $\iota_x \colon \{x\} \to U'$ is the inclusion. We see using the spectral sequence
\[ E_2^{p,q} = \cohH^p \iota_x^! \cohH^q (M^\bullet \vert_{U'}) \implies \cohH^{p+q}\iota_x^!(M^\bullet\vert_{U'}), \]
and the fact that $\iota_x$ is non-characteristic for each cohomology module $\cohH^q(M^\bullet \vert_{U'})$ (implying that the spectral sequence degenerates at $E_2$) that $p(M^\bullet\vert_{U'}) \geq j$, and so $p(j_*(M^\bullet\vert_{U'})) \geq j$, too.

Thus, we have shown that $p(M^\bullet) \geq j$.
\end{proof}

\begin{lemma} Let $M \in {\rm MHM}(Z)$.  Then for any $x\in Z$, we have
$\cohH^j i_x^! M = 0 \text{ for all } j > \dim Z$.
\end{lemma}
\begin{proof} Fix $x\in Z$. The claim is obvious if $\dim Z = 0$. 

As above, we will use the definition of mixed Hodge modules in \cite{OnTheDefinition}. We can replace $Z$ by a Zariski open neighborhood of $x$ in $Z$ because the question is local near $x$. Take such a neighborhood $U$ such that there exists a function $g\in \shO(U)$ with the property that on $U' = \{g\neq 0\}$, the module $M$ restricts to a variation of mixed Hodge structures.

We have the exact triangle
\[ i_* i^! M \to M \to j_*(M\vert_{U'}) \xrightarrow[]{+1}.\]

If $x\in U'$, then we have
\[ i_x^! M = i_x^! j_*(M\vert_{U'}) = i_{x,U'}^!(M\vert_{U'})\]
and the claim is obvious as $U'$ is smooth of dimension $\dim(Z)$.

If $x\notin U'$, then
\[ i_x^! i_* i^!M = i_x^! M\]
and so we can use induction on the dimension, using that each cohomology of $i_* i^! M$ is supported on $\{g=0\}$ which has strictly smaller dimension than $Z$. Then one uses induction on the dimension of the support and the spectral sequence:
\[ E_2^{i,j} = \cohH^i i_x^! \cohH^j(i_* i^! M) \implies \cohH^{i+j} i_x^! M.\]

Note that as $i$ is the inclusion of a divisor, $E_{2}^{i,j} \neq 0$ implies $j= 0, 1$. Thus, by induction on the dimension of the support, $E_2^{i,j}\neq 0$ implies $i+j \leq 1 + \dim \{g=0\} = \dim Z$.
\end{proof}

\begin{lemma} \label{lem-CohomologyBounds} Let $M^\bullet \in D^b({\rm MHM}(Z))$ be such that there exists $c\in \ZZ, \ell \in \ZZ_{\geq 0}$ such that
\[ \dim {\rm Supp}(\cohH^{i+c} M^\bullet) \leq \ell - i.\]
Then $\cohH^{j+c} i_x^! M^\bullet = 0$ for all $j > \ell$ and all $x\in Z$.
\end{lemma}
\begin{proof} By shifting $M^\bullet$ we can assume $c = 0$. By the spectral sequence
\[ E_2^{i,j} = \cohH^i i_x^! \cohH^j M^\bullet \implies \cohH^{i+j} i_x^! M^\bullet,\]
we want to show that $\cohH^i i_x^! \cohH^j M^\bullet = 0$ for all $i+j > \ell$. By the previous lemma and the assumption on $\dim {\rm Supp}(\cohH^j M^\bullet)$, we get the desired vanishing.
\end{proof}

\subsection{Poincar\'{e} duality and higher singularities} \label{subsect-PDHS}
Let $Z$ be a pure $d$-dimensional complex variety. The trivial Hodge module $\QQ^H_Z \in D^b({\rm MHM}(Z))$ satisfies the property that if $a_Z \colon Z \to {\rm pt}$ is the constant map, then
\[ \cohH^j (a_Z)_*\QQ^H_Z \in {\rm MHM}({\rm pt}) = {\rm MHS}\]
gives Deligne's mixed Hodge structure on the cohomology $H^j(Z,\QQ)$. Moreover,
\[ \cohH^j (a_Z)_! \QQ^H_Z \in {\rm MHS}\]
gives the canonical mixed Hodge structure on compactly supported cohomology $H^j_{\rm c}(Z,\QQ)$.

Although $\QQ_Z^H[d] \in D^b({\rm MHM}(Z))$ is not necessarily a single mixed Hodge module, it is known (for example, by comparing to the underlying perverse sheaf) that $\cohH^j( \QQ_Z^H[d]) \neq 0$ implies $j \leq 0$. Moreover, $\cohH^0(\QQ_Z^H[d]) \in {\rm MHM}(Z)$ satisfies the property that ${\rm Gr}^W_d \cohH^0(\QQ_Z^H[d])$ has underlying perverse sheaf equal to ${\rm IC}_Z$, the intersection complex perverse sheaf of \cite{BBDG}. By the weight formalism for mixed Hodge modules, we know that ${\rm Gr}^W_\ell \cohH^0(\QQ_Z^H[d]) \neq 0$ implies $\ell \leq d$. 

We write
\[{\rm IC}^H_Z = {\rm Gr}^W_d \cohH^0(\QQ_Z^H[d]),\]
and so in particular, there is a natural morphism
\[ \gamma_Z \colon \QQ^H_Z[d] \to {\rm IC}_Z^H,\]
which induces
\[ \cohH^0 \gamma_Z \colon \cohH^0(\QQ_Z^H[d]) \to {\rm IC}_Z^H.\]

As ${\rm IC}_X^H$ is pure, polarizable of weight $d$, there exists an isomorphism
\[ \mathbf D_Z ({\rm IC}_Z^H) \cong {\rm IC}_Z^H(d),\]
where $\mathbf D_Z$ is the duality of mixed Hodge modules on $Z$. On underlying perverse sheaves, it is Verdier duality.

By duality, we can also realize ${\rm IC}_Z^H$ as ${\rm Gr}^W_d(\cohH^0(\mathbf D_Z(\QQ^H_Z[d])(-d)))$. Again by the weight formalism, we see that ${\rm IC}_Z^H$ can be written as $W_d(\cohH^0(\mathbf D_Z(\QQ^H_Z[d])(-d)))$, and in particular, we have a monomorphism
\[ \cohH^0 \gamma_Z^\vee \colon {\rm IC}^H_Z \to \cohH^0(\mathbf D_Z(\QQ^H_Z[d])(-d)),\]
which, up to Tate twist, is dual to $\cohH^0 \gamma_Z$.

The composition $\gamma_Z^\vee \circ \gamma_Z$ gives a morphism
\[ \cohH^0(\QQ_Z^H[d]) \to \cohH^0(\mathbf D_Z(\QQ_Z^H[d])(-d)),\]
which uniquely determines a morphism in $D^b({\rm MHM}(Z))$:
\[ \psi_Z = \gamma_Z^\vee \circ \gamma_Z \colon \QQ_Z^H[d] \to \mathbf D_Z(\QQ_Z^H[d])(-d).\]

\begin{remark} \label{rmk-uniqueness} Saito \cite{SaitoMHM}*{(4.5.14)} shows that the morphism $\psi_Z$ is unique up to scalar multiplication on each irreducible component of $Z$.
\end{remark}

If the morphism $\psi_Z$ is a quasi-isomorphism, we say that the variety $Z$ is a \emph{rational} (or $\mathbb{Q}$-)\emph{homology manifold} (this notion is also called {\it rational smoothness} in the literature). This can be checked on underlying perverse sheaves, and so does not require the theory of Hodge modules.

The morphism $\psi_Z$ has recently found applications in the study of so-called ``higher singularities''. The reason for this comes from the connection with the Du Bois complex of the variety $Z$. We will not review the definition of the Du Bois complex here (aside from its connection to mixed Hodge modules), for such a review, see \cites{MustataPopaDuBois,SVV,PS}.

For $Z$ a pure $d$-dimensional variety, the $p$-th Du Bois complex is $\underline{\Omega}_Z^p \in D^b_{\rm coh}(\shO_Z)$. An important consequence of the definition is that there is a natural comparison morphism in $D^b_{\rm coh}(\shO_Z)$:
\[ \alpha_p \colon \Omega_Z^p \to \underline{\Omega}_Z^p,\]
for all $0 \leq p \leq \dim(Z)$, where $\Omega_Z^p$ is the sheaf of K\"{a}hler differentials. Then $\alpha_p$ is a quasi-isomorphism for all $p$ if $Z$ is smooth.

Recall that ${\rm Gr}^F_\bullet {\rm DR}_Z(\QQ_Z^H) \in D^b_{\rm coh}(\shO_Z)$ can be defined using local embeddings into smooth varieties. Then there is a natural quasi-isomorphism (\cite{SaitoMHC}):
\[ \underline{\Omega}_Z^p[d-p] \cong {\rm Gr}^F_{-p} {\rm DR}_Z(\QQ^H_Z[d]),\]
and so we can use the theory of mixed Hodge modules to try to understand these Du Bois complexes.

By applying Grothendieck duality $\mathbb D_Z(-) = R \mathcal Hom_{\shO_Z}(-,\omega_Z^\bullet)$, where $\omega_Z^\bullet$ is the dualizing complex on $Z$, we can write (using \propositionref{prop-GrDRProps})
\[ \mathbb D_Z(\underline{\Omega}_Z^p[d-p]) = {\rm Gr}^F_{p} {\rm DR}_Z(\mathbf D_Z(\QQ^H_Z[d])),\]
and using $d-p$ in place of $p$, we get
\[ \mathbb D_Z(\underline{\Omega}_Z^{d-p}[p]) = {\rm Gr}^F_{d-p} {\rm DR}_Z(\mathbf D_Z(\QQ_Z^H[d])) ={\rm Gr}^F_{-p} {\rm DR}_Z(\mathbf D_Z(\QQ_Z^H[d])(-d)).\]

Thus, applying ${\rm Gr}^F_{-p}{\rm DR}_Z$ to the morphism $\psi_Z$, we obtain natural morphisms
\[ \underline{\Omega}_Z^{p}[d-p] \cong {\rm Gr}^F_{-p}{\rm DR}_Z(\QQ^H_Z[d]) \to {\rm Gr}^F_{-p}{\rm DR}_Z(\mathbf D_Z(\QQ_Z^H[d])(-d)) \cong \mathbb D_Z(\underline{\Omega}_Z^{d-p}[p])\]
which are identified with $\phi^p [d-p]$ in the notation of \cite{FriedmanLaza} and the introduction.

We now define the classes of higher singularities, though we focus on the case when $Z$ has local complete intersection singularities.

\begin{definition}[\cites{JKSY,FriedmanLaza}] \label{def-HigherSings} A pure $d$-dimensional variety $Z$ with isolated or local complete intersection singularities has $k$-Du Bois singularities if for all $p\leq k$, the maps $\alpha_p \colon \Omega_Z^p \to \underline{\Omega}_Z^p$ are quasi-isomorphisms.

Such a variety $Z$ has $k$-rational singularities if it has $k$-Du Bois singularities and if, for all $p \leq k$, the maps $\phi^p\colon \underline{\Omega}_Z^p \to \mathbb D_Z(\underline{\Omega}_Z^{d-p})[-d]$ are quasi-isomorphisms. Equivalently, $Z$ has $k$-rational singularities if for all $p\leq k$, the composition $\phi^p \circ \alpha_p$ is a quasi-isomorphism.
\end{definition}

\begin{example} If $Z$ is smooth, then we mentioned above that $\alpha_p$ is a quasi-isomorphism for all $p$. But then the morphism
\[ \underline{\Omega}_Z^p \to \mathbb D_Z(\underline{\Omega}_Z^{d-p})[-d]\]
is the well-known isomorphism of locally free sheaves
\[ \Omega_Z^p \cong \mathcal Hom_{\shO_Z}(\Omega_Z^{d-p},\omega_Z).\]

Thus, smooth varieties are $k$-rational for all $k$.
\end{example}

\begin{remark}[Non-LCI Setting] \label{rmk-prek} In the non-local complete intersection case, these notions have been studied for isolated singularities by \cite{FLIsolated}, and in general by \cite{SVV} (see also \cites{SecantVarieties,Tighe} for some interesting examples).

There are several relevant notions: in \cite{SVV}, the definitions of $k$-Du Bois and $k$-rational that we gave above are called \emph{strict} $k$-Du Bois and \emph{strict} $k$-rational, respectively. However, comparing with the K\"{a}hler differentials is rather restrictive: there are not many examples of such singularities which are not local complete intersection.

Without comparing to the K\"{a}hler differentials, there are the notions of \emph{pre-}$k$-Du Bois and \emph{pre-}$k$-rational. Following \cite{SVV}, pre-$k$-Du Bois is the condition that $\cohH^i(\underline{\Omega}_Z^p) = 0$ for all $i > 0$ and $p\leq k$, and pre-$k$-rational is the condition that $\cohH^i(\mathbb D_Z(\underline{\Omega}_Z^{d-p})[-d]) = 0$ for all $i > 0$ and $p\leq k$.

Then a normal variety $Z$ is pre-$k$-rational if and only if it is pre-$k$-Du Bois and the maps $\phi^p \colon \underline{\Omega}_Z^p \to \mathbb D_Z(\underline{\Omega}_Z^{d-p})[-d]$ defined above are quasi-isomorphisms for all $p \leq k$.
\end{remark}

\section{Hodge Rational Homology Manifold Level} \label{sect-Definition}
\subsection{Definition and basic properties}
As mentioned above, a pure $d$-dimensional variety $Z$ is a \emph{rational} (or $\mathbb{Q}$-)\emph{homology manifold}, or \emph{rationally smooth}, if the morphism
\[ \psi_Z \colon \QQ^H_Z[d] \to (\mathbf D_Z \QQ^H_Z[d])(-d)\]
is a quasi-isomorphism. This is equivalent to requiring the map on the underlying $\QQ$-structure to be a quasi-isomorphism. This section is devoted to defining and studying a natural weakening of this notion.

We observed above that $\QQ^H_Z[d] \in D^{\leq 0}({\rm MHM}(Z))$, and by duality this implies $\mathbf D_Z(\QQ^H_Z[d]) \in D^{\geq 0}({\rm MHM}(Z))$. The following elementary lemma allows us to understand the map $\psi_Z$.

\begin{lemma} Let $\cA$ be an abelian category and let $A\in D^{\leq 0}(\cA), B\in D^{\geq 0}(\cA)$ be objects in the derived category. Let $\psi \colon A \to B$ be a morphism.

Then $\psi$ is a quasi-isomorphism if and only if $\cohH^0 \psi \colon \cohH^0A \to \cohH^0B$ is an isomorphism in $\cA$ and $\cohH^{-i} A = \cohH^i B = 0$ for all $i >0$.
\end{lemma}

Recall that we have factored $\cohH^0 \psi_Z = \cohH^0 \gamma_Z^\vee \circ \cohH^0 \gamma_Z$, where
\[ \gamma_Z \colon \QQ_Z^H[d] \to {\rm IC}_Z^H, \quad \gamma_Z^\vee = \mathbf D_Z(\gamma_Z)(-d).\]

In particular, by duality, we have the relation
\[ \mathbf D_Z(\psi_Z) = \psi_Z (d).\]

\begin{proposition} \label{prop-GammaProperties} We have the following:
\begin{itemize}
    \item The map $\cohH^0 \gamma_Z$ is an isomorphism if and only if it is a monomorphism.
    \item The map $\cohH^0 \gamma_Z^\vee$ is an isomorphism if and only if it is an epimorphism.
\end{itemize} 
Either condition is equivalent to $\cohH^0 \psi_Z$ being an isomorphism.

Thus, the variety $Z$ is a rational homology manifold if and only if $\QQ_Z[d]$ is perverse and either $\cohH^0 \gamma_Z$ or $\cohH^0 \gamma_Z^\vee$ is an isomorphism.
\end{proposition}
\begin{proof} The first three claims are immediate.

The condition that $\QQ_Z[d]$ is perverse is equivalent to saying that $\cohH^j (\QQ_Z^H[d]) = 0$ for all $j < 0$, and so the last claim follows by the previous lemma.
\end{proof}

We will now define the weakening of rational smoothness that is slightly different but equivalent to \definitionref{deff}, and compare its behavior to that when $Z$ is actually a rational homology manifold.

\begin{definition} Let $Z$ be a pure $d$-dimensional variety. We say $Z$ is a \emph{rational homology manifold to Hodge degree} $k$, or for short {\it $k$-Hodge rational homology} if the morphism
\[ {\rm Gr}^F_{-p}{\rm DR}_Z(\psi_Z) \colon {\rm Gr}^F_{-p}{\rm DR}_Z(\QQ^H_Z[d]) \to {\rm Gr}^F_{d-p}{\rm DR}_Z(\mathbf D_Z(\QQ^H_Z[d]))\]
is a quasi-isomorphism for all $p\leq k$. We set \[\HRH(Z)={\rm sup}\left\{k \in \ZZ_{\geq -1} \mid \textrm{$Z$ is a $k$-Hodge rational homology variety}\right\}\]
with $\HRH(Z)=-1$ if $Z$ is not $0$-Hodge rational homology. It admits a local version ${\rm HRH}_x(Z) = \max_{x\in U \subseteq Z} {\rm HRH}(U)$, where the maximum runs over Zariski open neighborhoods of $x$ in $Z$. 
\end{definition}

\begin{remark}\label{stark} By the discussion at the end of Section \ref{sect-Preliminaries}, this condition $\HRH(Z)\geq k$ is equivalent to having $\phi^p \colon \underline{\Omega}_Z^p \to \mathbb D_Z(\underline{\Omega}_Z^{d-p})[-d]$ be a quasi-isomorphism for all $p \leq k$. 

We will see in \remarkref{rmk-compareStark} that this notion is the same as the one studied in \cites{PPLefschetz,PSV}, where it is called condition $(D)_k$. 
\end{remark}

\begin{remark}
Let $f\colon\widetilde{Z}\to Z$ be a strong log resolution with reduced exceptional divisor $E$. Then $U:=Z\setminus Z_{\textrm{sing}}\cong \widetilde{Z}\setminus E$. Let $j\colon U\to Z$ and $j'\colon U\to \widetilde{Z}$ be the inclusions. The map $j_{!}\mathbb{Q}_U^H[d]\to \mathbb{Q}_Z^H[d]$ by duality yields \[{\bf D}_Z(\mathbb{Q}_Z^H[d])(-d)\to j_*\mathbb{Q}_U^H[d].\] The above fits into the commutative diagram:  
\begin{equation}
    \begin{tikzcd}
    \mathbb{Q}_Z^H[d]\arrow[d, "\psi_Z"]\arrow[r] & f_*\mathbb{Q}_{\widetilde{Z}}^H[d]\arrow[d]\\
    {\bf D}_Z(\mathbb{Q}_Z^H[d])(-d)\arrow[r] & j_*\mathbb{Q}_U^H[d]
\end{tikzcd}
\end{equation}
Applying ${\rm Gr}^F_{-k}{\rm DR}_Z$ and appropriate shifts, we obtain the commutative diagram:
\begin{equation}\label{pic1}
    \begin{tikzcd}
    \underline{\Omega}_Z^k\arrow[d, "\phi^k"]\arrow[r] & { R}f_*\Omega_{\widetilde{Z}}^k\arrow[d]\\
    {\mathbb D}_Z(\underline{\Omega}_Z^{d-k})[-d]\arrow[r] &{R}f_*{\Omega}_{\widetilde{Z}}^k(\log E)
\end{tikzcd}
\end{equation}
Note that the right vertical map is induced by the residue sequence on $\widetilde{Z}$ (it is an isomorphism for $k=0$). 
\end{remark}

To motivate the reader, we describe some useful properties of $k$-Hodge rational homology varieties that are proven in \cite{PPLefschetz} (thanks to the fact that $\HRH(Z)\geq k$ is equivalent to $Z$ satisfying $(D)_k$ in the sense of \cite{PPLefschetz}, see \remarkref{rmk-compareStark}): 

\begin{remark}\label{rmk-PPProperties}
Let $Z$ be a variety of pure dimension $d$. 
\begin{enumerate}
    \item $Z$ is a rational homology manifold if and only if $\HRH(Z)\geq k$ for all $k$ (i.e. $\HRH(Z)=+\infty$). 
    \item If $Z$ is quasi-projective with general hyperplane section $Z'$ and $\HRH(Z)\geq k$, then we also have $\HRH(Z') \geq k$.
    \item If $Z$ is normal and its singularities are pre-$k$-rational, then $\HRH(Z)\geq k$.
    \item Assume $\HRH(Z)\geq k$. 
    Then:
    \begin{itemize}
    \item[(a)] Its singularities are pre-$k$-Du Bois (resp.\ strict-$k$-Du Bois) if and only if they are pre-$k$-rational (resp.\ strict-$k$-rational).
    \item[(b)] $\phi^p$ is an isomorphism for $d-k-1\leq p\leq d$.
    \item[(c)] $\mathrm{codim}_Z(Z_{\textrm{nRS}})\geq 2k+3$ where $Z_{\textrm{nRS}}$ is the locus of $Z$ where it is not a rational homology manifold. In particular \[ \HRH(Z) < + \infty \text{ if and only if } \HRH(Z) \leq \frac{d-3}{2}.\]
    \item[(d)] $\mathrm{lcdef}(Z)\leq \max\left\{d-2k-3,0\right\}$ where $\textrm{lcdef}(Z)$ is the local cohomological defect of $Z$, given by
    \[ {\rm lcdef}(Z) = \max\{a \mid \cohH^{-a} \QQ_Z^H[d] \neq 0\}.\]
    \item[(e)] (Symmetry of Hodge-Du Bois numbers) The Hodge-Du Bois numbers $\underline{h}^{p,q}(Z):=\mathbb{H}^q(Z,\underline{\Omega}_Z^p)$ are partially equipped with the full symmetry:
    \[\underline{h}^{p,q}(Z)=\underline{h}^{q,p}(Z)=\underline{h}^{d-p,d-q}(Z)=\underline{h}^{d-q,d-p}(Z)\textrm{ for all $0\leq p\leq k,0\leq q\leq d$}.\]
\end{itemize}

\end{enumerate}    
\end{remark}

We end this subsection with some general remarks on the behavior of the invariant $\HRH(Z)$.

\begin{remark} \label{rmk-duality} By duality, we see that $\HRH(Z) \geq k$ if
\[ \mathbb D_Z({\rm Gr}^F_{-p} {\rm DR}_Z(\psi_Z)) = {\rm Gr}^F_{p} {\rm DR}_Z(\mathbf D(\psi_Z))\]
is a quasi-isomorphism for all $p\leq k$. As $\mathbf D_Z(\psi_Z) = \psi_Z(d)$, this is equivalent to having
\[ {\rm Gr}^F_{p-d} {\rm DR}_Z(\psi_Z)\]
be a quasi-isomorphism for all $p\leq k$.
\end{remark}


\subsection{Characterization via filtrations on local cohomology modules}\label{emb-sec}
The condition in \remarkref{rmk-duality} resembles the condition in \corollaryref{cor-technical}, though we cannot talk about $F_k \QQ^H_Z[d]$ without using a fixed local embedding. Indeed, the terms in these complexes are not $\shO_Z$-modules, and the morphisms are not $\shO_Z$-linear. To continue this discussion, we will assume $i\colon Z \hookrightarrow X$ is a closed embedding of $Z$ inside a smooth variety $X$. 

We focus on $i_* \psi_Z \colon i_* \QQ_Z^H[d] \to i_* \mathbf D_Z(\QQ_Z^H [d])(-d)$, which is a morphism of objects in $D^b({\rm MHM}(X))$. Then $\HRH(Z) \geq k$ if and only if ${\rm Gr}^F_{p-d} {\rm DR}_X(i_* \psi_Z)$ is a quasi-isomorphism for all $p\leq k$. By \corollaryref{cor-technical}, this is equivalent to $F_{k-d} i_* \psi_Z$ being a quasi-isomorphism. 

This condition is equivalent to $F_{k-d} \cohH^{-j}(i_*\QQ^H_Z[d]) = F_{k-d} \cohH^{j}(i_* \mathbf D_Z(\QQ^H_Z[d])(-d)) = 0$ for $j > 0$ and $F_{k-d} \cohH^0(i_* \psi_Z)$ being an isomorphism. From here, we can give the lower bound on $\HRH(Z)$ in terms of local cohomology in \theoremref{thm-characterize}\eqref{thmmain}. Recall that we are indexing our Hodge filtrations following the conventions for right $\Dmod$-modules.

\begin{remark}
In \cite{CDM}, it is shown that when $Z$ is a complete intersection, then 
\[F_{k-n}W_{n+c} \cohH^c_Z(\shO_X) = F_{k-n} \cohH^c_Z(\shO_X) = P_k\cohH^c_Z(\shO_X)
\] is equivalent to $Z$ having $k$-rational singularities. Here $P_k \cohH^c_Z(\shO_X)$ is the \emph{pole order filtration}, consisting of elements which are annihilated by $\mathcal I_Z^{k+1}$, where $\mathcal I_Z \subseteq \shO_X$ is the ideal sheaf defining $Z$ in $X$. The comparison with the pole order filtration is essentially the same as the map $\alpha_k$ comparing the K\"{a}hler differentials to the Du Bois complex.
\end{remark}

\begin{proof}[Proof of \theoremref{thm-characterize}\eqref{thmmain}] We have $d = n-c$ by definition of the codimension of $Z$ in $X$, and so we are interested in
\[ i_* \QQ_Z^H[n-c] \to i_* \mathbf D_Z(\QQ_Z^H[n-c])(c-n).\]

By definition, $\QQ_Z^H[n-c] = i^*(\QQ_X^H[n])[-c]$. Applying duality and using the fact that $X$ is smooth (so that $\QQ_X^H[n]$ is pure, polarizable of weight $n$), we get
\[ \mathbf D_Z(\QQ_Z^H[n-c]) = i^!(\mathbf D_X(\QQ_X^H[n]))[c] = i^!(\QQ_X^H[n](n))[c]\]
so that $i_*\psi_Z$ can be identified with the morphism
\[ i_*\psi_Z\colon i_*\QQ_Z^H[n-c] \to i_* i^!(\QQ_X^H[n])[c](c).\]

Using that $\cohH^j(i_*i^!(\QQ_X^H[n])) = \cohH^j_Z(\shO_X)$ as mixed Hodge modules, we see that
\[ F_{k-d} \cohH^j(i_* i^!(\QQ_X^H[n])(c)) = F_{k-d-c} \cohH^j_Z(\shO_X) = F_{k-n} \cohH^j_Z(\shO_X).\]

Thus, we see that $\HRH(Z) \geq k$ if and only if $F_{k-d} i_* \psi_Z$ is a quasi-isomorphism if and only if for all $j>0$ we have $F_{k-d}\cohH^{-j}(i_* \QQ_Z^H[n-c]) = F_{k-d}\cohH^{j}(i_*i^!(\QQ_X^H[n])(c)) = 0$ and $F_{k-d} \cohH^0(\psi_Z)$ is an isomorphism. As $F_{k-d}\cohH^0(\psi_Z) = F_{k-d} \cohH^0 \gamma_Z^\vee \circ F_{k-d} \cohH^0 \gamma_Z$, the last condition is equivalent to both $F_{k-d}\cohH^0 \gamma_Z^\vee$ and $F_{k-d}\cohH^0 \gamma_Z$ being isomorphisms by \propositionref{prop-GammaProperties}.

Note that $F_{k-d}\cohH^0 \gamma_Z^\vee$ is an isomorphism if and only if we have equality
\[ F_{k-n} W_{n+c} \cohH^c_Z(\shO_X) = F_{k-n} \cohH^c_Z(\shO_X),\]
and so we see that $\HRH(Z)\geq k$ implies the conditions in the theorem statement.

For the converse, we assume \[ F_{k-n} \cohH^{j}_Z(\shO_X) = 0\text{ for all } j > c, \quad F_{k-n} W_{n+c} \cohH^c_Z(\shO_X) = F_{k-n} \cohH^c_Z(\shO_X),\]
and we want to show $\HRH(Z) \geq k$.

For ease of notation, let $A = \QQ_Z^H[d]$ and let $B = \mathbf D(A)(-d)$. By the discussion at the beginning of the proof, to show that $\HRH(Z)\geq k$ it suffices under our assumption to show that $F_{k-d} \cohH^{-j}(A) = 0$ for all $j > 0$ and that $F_{k-d}\cohH^0 \gamma_Z$ is an isomorphism. For this, we follow the argument of \cite{CDM}*{Lem. 3.5}. To prove $F_{k-d} \cohH^{-j}(A) = 0$ it suffices to prove for all $\ell \in \ZZ$ that $F_{k-d} {\rm Gr}^W_\ell \cohH^{-j}(A) = 0$. To prove that $F_{k-d}\cohH^0 \gamma_Z$ is an isomorphism it suffices to prove that $F_{k-d} {\rm Gr}^W_{\ell}(\cohH^0(A)) = 0$ for all $\ell < d$.

Our assumption is equivalent to $F_{k-d} \cohH^j(B) = 0$ for all $j > 0$ and that $F_{k-d} {\rm Gr}^W_{\ell} \cohH^0(B) =0$ for all $\ell > d$.

By polarizability of the pure Hodge module ${\rm Gr}^W_\ell \cohH^j(B)$, we have an isomorphism
\[ \mathbf D({\rm Gr}^W_\ell \cohH^j B) \cong ({\rm Gr}^W_{\ell} \cohH^j B)(\ell).\] But we also have
\[ \mathbf D({\rm Gr}^W_\ell \cohH^j B) \cong {\rm Gr}^W_{-\ell} \cohH^{-j} (\mathbf D(B)).\]

We have that $\mathbf D(B) \cong A(d)$, and so if we apply $F_p$ to this isomorphism (keeping in mind the Tate twists), we have
\[ F_{p-\ell} {\rm Gr}^W_\ell \cohH^j B \cong F_{p-d} {\rm Gr}^W_{2d - \ell} \cohH^{-j} A.\]

By the weight formalism, we know that ${\rm Gr}^W_{\ell} \cohH^j (B) \neq 0$ implies $\ell \geq d + j$ (and the same inequality is implied when ${\rm Gr}^W_{2d-\ell} \cohH^{-j}(A) \neq 0$). For $j > 0$, we trivially have $\ell > d$, and for $j = 0$, we only need to consider $\ell > d$. Plugging in $p = k-d +\ell$, we get
\[ 0 = F_{k-d} {\rm Gr}^W_{\ell} \cohH^j(B) = F_{k-2d+\ell} ({\rm Gr}^W_{2d-\ell}\cohH^{-j}(A)),\]
and so because $W$ is exhaustive (for the $j>0$ case) and $\ell >d$, this gives the desired vanishing.
\end{proof}

\begin{remark} The theorem is essentially saying that $\HRH(Z)$ is controlled by the morphism $F_{k-d}i_* \gamma_Z^\vee$.

Although duality is used in the argument, it is important to note that it is not equivalent to study $F_{k-d}i_*\gamma_Z$. Indeed, in the non-rational homology manifold hypersurface case, we will see below that $F_{k-d}i_*\gamma_Z$ can be an isomorphism when $F_{k-d} i_*\gamma_Z^\vee$ is not.
\end{remark}

\begin{remark} \label{rmk-compareStark} At this point, we can see that $\HRH(Z) \geq k$ if and only if $Z$ satisfies the condition $D_k$ of \cites{PPLefschetz,PSV}. Indeed, as both notions are local, we can assume $i\colon Z \to X$ is a closed embedding into a smooth variety $X$. Then by the argument above, we have
\[ \HRH(Z) \geq k \text{ if and only if } F_{k-d} i_* \gamma_Z^\vee \text{ is a quasi-isomorphism},\]
which is true if and only if
\[{\rm Gr}^F_{p-d} {\rm DR}_X(i_* \gamma_Z^\vee) \text{ is a quasi-isomorphism for all } p \leq k.\]

By duality, this is equivalent to
\[ {\rm Gr}^F_{d-p} {\rm DR}_X(i_* \mathbf D(\gamma_Z^\vee)) \text{ being a quasi-isomorphism for all } p \leq k,\]
and finally, using that $\mathbf D(\gamma_Z^\vee) = \gamma_Z(d)$, we have that this is equivalent to the natural map
\[ {\rm Gr}^F_{-p} {\rm DR}_X(i_* \QQ_Z^H[d]) \to {\rm Gr}^F_{-p} {\rm DR}_X(i_* {\rm IC}_Z^H) \text{ being a quasi-isomorphism for all } p\leq k,\]
which is the condition $D_k$.
\end{remark}

\begin{proof}[Proof of \corollaryref{cor-DANVanishing}]
Using Kodaira-Saito vanishing \cite{SaitoMHM}*{Prop. 2.33} for the Hodge module ${\rm IC}_Z^H$ we obtain 
\[ \mathbb H^q(Z, {\rm I}\underline{\Omega}_Z^p \otimes_{\shO_Z} L^{-1}) = 0 \text{ for all } p+q < d.\]
The conclusion follows immediately from the above remark.
\end{proof}

\begin{proof}[Proof of \corollaryref{cor-rational}] Assume $Z$ is a Cohen-Macaulay subvariety of $X$ of pure codimension $c$.

Recall the notation of the corollary statement: the filtration $E_\bullet \cohH^q_Z(\shO_X)$ is defined by
\[ E_\bullet \cohH^q_Z(\shO_X) = {\rm Im}\left[\mathcal E xt^q(\shO_X/I_Z^{\bullet+1},\shO_X)\to \cohH^q_Z(\shO_X)\right],\]
where $I_Z\subseteq \shO_X$ is the ideal sheaf defining $Z$ in $X$.

Under the Cohen-Macaulay assumption, we get $F_{-n}\cohH^q_Z(\shO_X) = 0$ for $q >c$. Moreover, the result of \cite{MustataPopaDuBois}*{Thm. C} says that, under the Cohen-Macaulay assumption, $Z$ is Du Bois if and only if it satisfies $F_{-n}\cohH^c_Z(\shO_X)= E_0 \cohH^c_Z(\shO_X)$. 

Thus, either assumption in the corollary statement implies that $Z$ is Du Bois, so we can assume $Z$ is Du Bois. Then $Z$ has rational singularities if and only if $\HRH(Z) \geq 0$, and so the claim follows from the previous theorem.
\end{proof}

Let $D$ be a hypersurface inside a smooth variety $X$. The notion of Hodge ideals $I_p(D)$ was introduced in \cites{HI}, and subsequently their weighted variants $I_p^{W_\ell}(D)$ were introduced in \cites{WHI}. It is interesting to note that $\HRH(D)$ can be detected through the weighted Hodge ideals:

\begin{corollary}\label{cor-WHI} Let $D$ be a hypersurface inside a smooth variety $X$ of dimension $n$. Then $\HRH(D)\geq k$ if and only if $I_p^{W_1}(D)=I_p(D)$ for all $0\leq p\leq k$. 
\end{corollary}
\begin{proof}
Observe that $\HRH(D) \geq k$ if and only if $F_{p-n}{\rm Gr}^W_{n+\ell} \cohH^{1}_D(\shO_X)=0$ for $0\leq p\leq k$, $\ell \geq 2$ by \theoremref{thm-characterize}\eqref{thmmain}. The assertion is an immediate consequence of the exact sequence
\[0\to I_p^{W_{\ell-1}}(D)\otimes \shO_X((p+1)D)\to I_p^{W_{\ell}}(D)\otimes \shO_X((p+1)D)\to F_{p-n}{\rm Gr}^W_{n+\ell} \shO_X(*D)\to 0\]
given by \cite{WHI}*{(6.1)}, and the fact $\cohH^{1}_D(\shO_X)=\shO_X(*D)/\shO_X$.
\end{proof}

In the case of isolated singularities, the connection between this value and the lack of Poincar\'{e} duality of $D$ was noted in \cite{WHI}*{Remark 6.8}.

We can also now show how the invariant ${\rm HRH}(-)$ behaves under Cartesian product.

Recall the notation $p(A,F) = \min\{p\mid F_p A \neq 0\}$ for any $(A,F)$ where $F_\bullet A$ is a bounded below filtration.

\begin{lemma} \label{lem-ProductPartialSmooth} Let $Z_1,Z_2$ be two pure dimensional complex algebraic varieties. Then
\[\HRH(Z_1 \times Z_2) = \min\{ \HRH(Z_1),\HRH(Z_2)\}.\]
\end{lemma}

Before beginning the proof, we recall the structure of the external product for (complexes of) mixed Hodge modules.

 For any two mixed Hodge modules $M_i$ on smooth varieties $X_i$, the underlying filtered $\Dmod$-module of $M_1\boxtimes M_2$ is $\cM_1\boxtimes \cM_2$ with convolution Hodge filtration
\[ F_p (\cM_1 \boxtimes \cM_2) = \sum_{i+j = p} F_i \cM_1 \boxtimes F_j \cM_2,\]
and so
\[{\rm Gr}^F_p (\cM_1 \boxtimes \cM_2) = \bigoplus_{i+j = p} {\rm Gr}^F_i \cM_1 \boxtimes {\rm Gr}^F_j \cM_2.\]

Given two complexes $M_i^\bullet \in D^b({\rm MHM}(X_i))$, we have
\[ \cohH^k( M_1^\bullet \boxtimes M_2^\bullet) = \bigoplus_{i+j = k} \cohH^i M_1^\bullet \boxtimes \cohH^j M_2^\bullet.\]

Thus, we have
\[ {\rm Gr}^F_p \cohH^k(\cM_1^\bullet \boxtimes \cM_2^\bullet) = \bigoplus_{a+b=p} \bigoplus_{i+j = k} {\rm Gr}^F_a \cohH^i(\cM_1^\bullet) \boxtimes {\rm Gr}^F_b \cohH^j(\cM_2^\bullet).\]

Now, for $Z_1,Z_2$ not necessarily smooth, consider the exact triangle in $D^b({\rm MHM}(Z_i))$:
\[ \QQ_{Z_i}^H[d_i] \to {\rm IC}_{Z_i}^H \to \cK_{Z_i}^\bullet[1] \xrightarrow[]{+1}.\]

It is easy to see that we have
\[ \QQ_{Z_1}^H[d_1] \boxtimes \QQ_{Z_2}^H[d_2] \cong \QQ_{Z_1\times Z_2}^H[d_1+d_2]\]
\[ {\rm IC}_{Z_1}^H \boxtimes {\rm IC}_{Z_2}^H \cong {\rm IC}_{Z_1 \times Z_2}^H.\]

We make use of the following easy lemma, to relate the RHM defect objects.

\begin{lemma} \label{lem-octahedral} For $i = 1,2$, let $A_i \xrightarrow[]{\alpha_i} B_i \xrightarrow[]{\beta_i} C_i \xrightarrow[]{+1}$ be an exact triangle. Let $S = {\rm cone}(\alpha_1 \boxtimes \alpha_2)$. Then there are exact triangles
\[ A_1 \boxtimes C_2 \to S \to C_1 \boxtimes B_2 \xrightarrow[]{+1}\]
\[ C_1 \boxtimes A_2 \to S \to B_1 \boxtimes C_2 \xrightarrow[]{+1},\]
such that the compositions
\[ A_1 \boxtimes C_2 \to S \to B_1 \boxtimes C_2,\]
\[ C_1 \boxtimes A_2 \to S \to C_1 \boxtimes B_2\]
can be identified with $\alpha_1 \boxtimes {\rm id}_{C_2}$ and ${\rm id}_{C_1} \boxtimes \alpha_2$, respectively.
\end{lemma}
\begin{proof} This follows from the octahedral axiom and the exactness of the functor $ - \boxtimes M$ for any object $M$.
\end{proof}
\begin{proof}[Proof of \lemmaref{lem-ProductPartialSmooth}]
By the local nature of the claim, we can assume without loss of generality that both $Z_i$ are embeddable.

We consider the triangles
\[ \QQ_{Z_1}^H[d_1] \xrightarrow[]{\chi_1} {\rm IC}_{Z_1}^H \to \cK_{Z_1}^\bullet[1] \xrightarrow[]{+1},\]
\[ \QQ_{Z_2}^H[d_2] \xrightarrow[]{\chi_2} {\rm IC}_{Z_2}^H \to \cK_{Z_2}^\bullet[1] \xrightarrow[]{+1},\]
\[ \QQ_{Z_1 \times Z_2}^H[d_1+d_2] \xrightarrow[]{\chi_{12}} {\rm IC}_{Z_1 \times Z_2}^H \to \cK_{Z_1 \times Z_2}^\bullet[1] \xrightarrow[]{+1}.\]

By uniqueness of the $\chi$-morphisms (using an argument similar to \cite{SaitoMHM}*{(4.5.14)}), we get $\chi_1 \boxtimes \chi_2 = \chi_{12}$ up to non-zero scalar multiples on the connected components, so that we can identify
\[ {\rm cone}(\chi_1 \boxtimes \chi_2) \cong \cK_{Z_1\times Z_2}^\bullet[1].\]

By \lemmaref{lem-octahedral}, we have (after shifting by $[1]$) exact triangles
\[ \QQ_{Z_1}^H[d_1] \boxtimes \cK_{Z_2}^\bullet \to \cK_{Z_1\times Z_2}^\bullet \to \cK_{Z_1}^\bullet \boxtimes {\rm IC}_{Z_2}^H \xrightarrow[]{+1},\]

\[ \cK_{Z_1}^\bullet \boxtimes \QQ_{Z_2}^H[d_2] \to \cK_{Z_1\times Z_2}^\bullet \to {\rm IC}_{Z_1}^H \boxtimes \cK_{Z_2}^\bullet \xrightarrow[]{+1}.\]

Now, apply to these triangles the exact functor ${\rm Gr}^F_{-p} {\rm DR}(-)$. Using that $Z_1$ and $Z_2$ are both embeddable, we have the quasi-isomorphisms
\begin{equation} \label{eq-BoxTimesDecomposition} {\rm Gr}^F_{-p}{\rm DR}(M^\bullet \boxtimes N^\bullet) \cong \bigoplus_{i+j = p} {\rm Gr}^F_{-i} {\rm DR}(M^\bullet) \boxtimes {\rm Gr}^F_{-j} {\rm DR}(N^\bullet),\end{equation}
for any $M_i^\bullet \in D^b({\rm MHM}(Z_i))$, see \cite{LankVenkatesh}*{(B.1)}.

Let $\delta_i = {\rm HRH}(Z_i)$. We want to prove that ${\rm Gr}^F_{-p} {\rm DR}(\cK_{Z_1\times Z_2}^\bullet) = 0$ if and only if $p \leq \min\{\delta_1,\delta_2\}$. Without loss of generality, assume this minimum is equal to $\delta_1$.

First, assume $p\leq \delta_1$. Then we have
\[ {\rm Gr}^F_{-p} {\rm DR}(\QQ_{Z_1}^H[d_1] \boxtimes \cK_{Z_2}^\bullet) = \bigoplus_{i+j=p} {\rm Gr}^F_{-i} {\rm DR}(\QQ_{Z_1}^H[d_1]) \boxtimes {\rm Gr}^F_{-j} {\rm DR}(\cK_{Z_2}^\bullet).\]

Using that the first term is non-zero if and only if $i \geq 0$, we have the inequality $j \leq i+j \leq p \leq \delta_1 \leq \delta_2$, and so the second factor on the right is zero for any such $j$. Similarly, we conclude that ${\rm Gr}^F_{-p} {\rm DR}(\cK_{Z_1}^\bullet \boxtimes {\rm IC}_{Z_2}^H) = 0$, hence ${\rm Gr}^F_{-p} {\rm DR}(\cK_{Z_1\times Z_2}^\bullet) = 0$, as claimed.

Conversely, assume we have vanishing ${\rm Gr}^F_{-p} {\rm DR}(\cK_{Z_1\times Z_2}^\bullet) = 0$. We want to prove that $p \leq \delta_1$. The composition
\[ \cK_{Z_1}^\bullet \boxtimes \QQ_{Z_2}^H[d_2] \to \cK_{Z_1\times Z_2}^\bullet \to \cK_{Z_1}^\bullet \boxtimes  {\rm IC}_{Z_2}^H\]
is identified with ${\rm id}_{\cK_{Z_1}^\bullet} \boxtimes \chi_2$. When we apply ${\rm Gr}^F_{-p} {\rm DR}(-)$, this composition is zero, because it factors through ${\rm Gr}^F_{-p}{\rm DR}(\cK_{Z_1\times Z_2}^\bullet) = 0$.

Decomposing along the direct sum \eqref{eq-BoxTimesDecomposition} and using the fact that the morphism
\[ {\rm Gr}^F_{-j}{\rm DR}(\QQ_{Z_2}^H[d_2]) \to {\rm Gr}^F_{-j}{\rm DR}({\rm IC}_{Z_2}^H)\]
is never zero (it restricts to the identity on the smooth locus of $Z_2$), we conclude that
\[ {\rm Gr}^F_{-i} {\rm DR}(\cK_{Z_1}^\bullet) = 0\]
for all $i \leq p$, which proves $p \leq \delta_1 = {\rm HRH}(Z_1)$, as claimed.
\end{proof}

For example, if $Z_2$ is a rational homology manifold, then $\HRH(Z_1\times Z_2) = \HRH(Z_1)$. Even for smooth morphisms which are not projections from a product, we will see in \lemmaref{lem-kPRSSmoothPullback} that $\HRH$ is preserved.

\subsection{Characterization via local cohomology at a point and links}\label{link-sec}
We can now relate $\HRH(Z)$ to the local cohomology at a point $x\in Z$. Using the results of \cite{DurfeeSaito}, we relate $\HRH(Z)$ to the mixed Hodge structure on the cohomology of the link.

Let $S^\bullet= {\rm cone}(\psi_Z)$, where $\psi_Z\colon \QQ^H_Z[d] \to \mathbf D_Z(\QQ^H_Z[d])(-d)$. If $Z$ is embeddable into a smooth variety $X$, then we saw above that $\HRH(Z)\geq k$ if and only if $F_{k-d}\psi_Z$ is a quasi-isomorphism, which holds (by strictness of morphisms with respect to the Hodge filtration) if and only if $F_{k-d}S^\bullet = 0$. 

Let $i_x \colon \{x\} \to Z$ be the inclusion of a point. The local cohomology of $Z$ at $x$ is given by $\cohH^k(i_x^!\QQ^H_Z)$ and denoted $H^k_{\{x\}}(Z)$. For any neighborhood $x\in V\subseteq Z$, if $i_{x,V} \colon \{x\} \to V$ is the inclusion, then $i_x^! \QQ^H_Z = i_{x,V}^! \QQ^H_V$. Thus, as far as local cohomology is concerned, we can replace $Z$ with any open neighborhood of $x$. In particular, we can choose an affine neighborhood, so that we can assume $Z$ is embeddable into a smooth variety. We will need the following lemma:

\begin{lemma} \label{lem-pullbackHodge} Let $\phi \colon M^\bullet \to N^\bullet$ be a morphism in $D^b({\rm MHM}(X))$ where $X$ is a smooth algebraic variety. Let $i\colon Y\to X$ be a closed embedding. Assume $F_k \phi$ is a quasi-isomorphism. Then $F_k i_*i^!(\phi), F_k i_* i^*(\phi)$ are quasi-isomorphisms. If $Y$ is a smooth subvariety, then $F_k i^!(\phi),F_k i^*(\phi)$ are quasi-isomorphisms.
\end{lemma}
\begin{proof}  It suffices to prove the following: let $C^\bullet \in D^b({\rm MHM}(X))$ be such that $F_k C^\bullet$ is acyclic. Then $F_k i_* i^*(C^\bullet)$ and $F_k i_*i^!(C^\bullet)$ are acyclic. Indeed, this implies the lemma by applying this claim to the cone of $\phi$ and using that $i_* i^!,i_*i^*$ are exact functors between triangulated categories.

Again, using that $i_* i^!$ and $i_*i^*$ are exact functors, this reduces to the claim when $C$ is a single mixed Hodge module. Indeed, assume $C^\bullet \in D^{[a,b]}({\rm MHM}(X))$ and we use induction on $b-a$. We have the natural exact triangle
\[ \tau_{\leq b-1}C^\bullet \to C^\bullet \to \cohH^b(C^\bullet)[-b] \xrightarrow[]{+1},\]
where, by assumption, $F_k$ applied to any term in the triangle is acyclic. Induction handles the outer two terms. So we can assume $C^\bullet = C \in {\rm MHM}(X)$.

As this is a local statement, we can assume $Y= V(f_1,\dots, f_r)\subseteq X$ for some $f_1,\dots, f_r \in \shO_X(X)$.

Let $\Gamma \colon X\to X\times \mathbb A^r_t$ be the graph embedding along $f_1,\dots, f_r$. If $\sigma\colon X\times \{0\}\to  X\times\mathbb A^r_t$ is the inclusion of the zero section, then by Base Change (\exampleref{eg-BaseChange}) we have isomorphisms $\sigma^*\Gamma_* \cong i_*i^*, \sigma^!\Gamma_* \cong i_*i^!$. Moreover, we know that $F_k \Gamma_*(C) = 0$ by definition of the direct image for mixed Hodge modules (recall that we index like right $\Dmod$-modules). Thus, we have reduced to the case that $Y\subseteq X$ is a smooth subvariety defined by $t_1,\dots, t_r$. The claim then follows by \cite{CD}*{Thm. 1.2} (or \cite{CDS}*{Thm. 1}).
\end{proof}

Now we can prove the connection with the local description of being a rational homology manifold using local cohomology at $x\in Z$. We note that this invariant is related to the question of whether $H^{2d}_{\{x\}}(Z)$ is one dimensional, which by a result of Brion \cite{Brion}*{Prop. A1} is equivalent to $Z$ being irreducible near $x$. 

First, we prove a lemma which strengthens \cite{PPLefschetz}*{Prop. 6.4}. It is simply a corollary of \cite{PPLefschetz}*{Cor. 7.5}. We have the following exact triangles considering the RHM defect object and its dual:
\[ K_Z^\bullet \to \QQ_Z^H[d] \xrightarrow[]{\gamma_Z} {\rm IC}_Z^H \xrightarrow[]{+1},\]
\[ {\rm IC}_Z^H \xrightarrow[]{\gamma_Z^\vee} \mathbf D_Z(\QQ_Z^H[d])(-d) \to \mathbf D_Z(K_Z^\bullet)(-d) \xrightarrow[]{+1},\]
\[ \QQ_Z^H[d] \xrightarrow[]{\psi_Z}  \mathbf D_Z(\QQ_Z^H[d])(-d) \to S^\bullet \xrightarrow[]{+1}.\]

The octahedral axiom gives an exact triangle
\[ K_Z^\bullet[1] \to S^\bullet \to \mathbf D_Z(K_Z^\bullet)(-d) \xrightarrow[]{+1},\]
and so because the first object has non-zero cohomology only in strictly negative degrees and the third object only has non-zero cohomology in non-negative degrees, we have isomorphisms
\[ \cohH^{-i}K_Z^\bullet \cong \cohH^{-i-1}S^\bullet \text{ for all } i > 0, \quad \cohH^i \mathbf D_Z(K_Z^\bullet)(-d) \cong \cohH^i S^\bullet \text{ for all } i\geq 0.\]

\begin{lemma} \label{lem-SupportDimBound} We have the following inequality:
\[ \dim {\rm Supp}(\cohH^{-i} K_Z^\bullet) = \dim {\rm Supp}(\cohH^i \mathbf D_Z(K_Z^\bullet)(-d)) \leq d - 2{\rm HRH}(Z)-3 - i.\]

In particular, $\dim {\rm Supp}(\cohH^i S^\bullet) \leq d- 2\HRH(Z) -3 -i$.
\end{lemma}
\begin{proof} For the first inequality, the equality on the left holds by duality. We prove the first inequality for $K_Z^\bullet$. For ease of notation let $k= \HRH(Z)$.

First of all, \cite{PPLefschetz}*{Cor. 7.5(ii)} shows that
\[ \cohH^{2k+2-d} K_Z^\bullet = 0.\]

The claim is local, so we can assume $Z$ is quasi-projective. Note that ${\rm HRH}(Z)$ does not decrease under restriction to a general hyperplane $L$. Using that $\iota^* K_Z^\bullet = K_{Z\cap L}^\bullet[1]$ as shown in \cite{PPLefschetz}*{Lem. 6.6}, the argument of \cite{PPLefschetz}*{Prop. 6.4} gives the first claim.

The dimension bound on the cohomology of $S$ follows when $i \geq 0$ using $\cohH^i \mathbf D_Z(K_Z^\bullet)(-d) \cong \cohH^i S^\bullet \text{ for all } i\geq 0$. For the negative cohomologies, the desired bound is
\[ \dim {\rm Supp}(\cohH^{-i}S^\bullet) \leq d - 2\HRH(Z) -3 +i,\]
and so follows from the isomorphism
\[\cohH^{-i}S^\bullet \cong \cohH^{-(i-1)}K_Z^\bullet,\]
which by the first inequality has dimension less than or equal to
\[d - 2\HRH(Z) -3 -(i-1) \leq d-2\HRH(Z) -3 +i,\] using that $i> 0$.
\end{proof}

Recall that the local cohomological defect is given by
\[ {\rm lcdef}(Z) = \max\{a \mid \cohH^{-a}(\QQ^H_Z[d]) \neq 0\}.\] 
It admits a local version ${\rm lcdef}_x(Z) = \min_{x\in U \subseteq Z} {\rm lcdef}(U)$, where the minimum runs over Zariski open neighborhoods of $x$ in $Z$.

\begin{remark} \label{rmk-MPUpperBound} By the Hartshorne-Lichtenbaum Theorem \cite{Hartshorne}*{Thm. 3.1} \cite{Ogus}*{Cor. 2.10}, which has also been recovered by Musta\c{t}\u{a}-Popa in \cite{MustataPopaDuBois}*{Cor. 11.9}, we have the following bound for the local cohomological defect:
\[ {\rm lcdef}(Z) \leq d-1 \text{ if and only if no irreducible component of }Z \text{ is a point.}\]
As we are working with $Z$ purely $d$-dimensional with $d >0$, this is automatic for us.
\end{remark}

We begin with the following observation in the isolated singularities setting, which is a warm-up to the proof of \theoremref{thm-TopLocCoh}:
\begin{lemma} \label{lem-Irreducible} Let $x\in Z$ be such that $Z\setminus \{x\}$ is a rational homology manifold. Then $Z$ is irreducible near $x$ if and only if ${\rm lcdef}_x(Z) \leq d-2$.
\end{lemma}
\begin{proof} If $Z$ is a rational homology manifold, then \cite{Brion}*{Prop. A1} shows that $Z$ is irreducible. Moreover, ${\rm lcdef}(Z) = 0$ in this case. So we assume $Z$ is not a rational homology manifold at $x$.

Let $i_x \colon \{x\} \to Z$ be the inclusion. Apply $i_x^!$ to the triangle
\[ \QQ_Z^H[d] \to \mathbf D_Z(\QQ_Z^H[d])(-d) \to S^\bullet \xrightarrow[]{+1}\]
to get
\[ i_x^! \QQ_Z^H[d] \to \QQ^H(-d)[-d] \to i_x^! S^\bullet \xrightarrow[]{+1}.\]

This gives the exact sequence
\[ 0 \to \cohH^{d-1} i_x^! S^\bullet \to H^{2d}_{x}(Z) \to \QQ(-d) \to \cohH^{d} i_x^! S^\bullet \to 0,\]
and so it suffices to prove that $\cohH^{d-1}i_x^! S^\bullet = \cohH^{d} i_x^! S^\bullet = 0$ if and only if ${\rm lcdef}(Z) \leq d-2$.

As $Z\setminus \{x\}$ is a rational homology manifold, we know $S^\bullet$ is supported at $x$. Thus, we can write $S^\bullet = i_{x*}S'$ for some $S' \in D^b( {\rm MHM}(\{x\}))$ and so $i_x^! S^\bullet = S'$. Thus, it suffices to prove that
\[ \cohH^{d-1} S' =\cohH^{d}S' = 0 \text{ if and only if } {\rm lcdef}(Z) \leq d-2.\]

Again, as $i_{x*}S' = S^\bullet$, it suffices to consider $\cohH^\bullet S^\bullet$ itself. But then by definition (using that $S^\bullet\neq 0$), we have
\[ {\rm lcdef}(Z) = \max\{a \mid \cohH^{a}(\mathbf D_Z(K_Z^\bullet)) \neq 0\} = \max\{a \mid \cohH^{a}S^\bullet \neq 0 \},\]
so the claim follows.
\end{proof}

In the non-isolated setting, we see that $\HRH_x(Z)$ can also guarantee irreducibility, as well as vanishing of some higher local cohomology spaces of $Z$ at $x$. 

\begin{proof}[Proof of \theoremref{thm-TopLocCoh}] Apply $i_x^!$ to the triangle
\[ \QQ_Z^H[d] \to \mathbf D_Z (\QQ_Z^H[d])(-d) \to S^\bullet \xrightarrow[]{+1},\]
which gives
\[ i_x^! \QQ_Z^H[d] \to \QQ^H(-d) [-d] \to i_x^! S^\bullet \xrightarrow[]{+1},\]
and thus, by the long exact sequence in cohomology, we get an exact sequence
\[ 0 \to \cohH^{d-1} i_x^! S^\bullet \to H_{\{x\}}^{2d}(Z) \to \QQ(-d) \to \cohH^{d} i_x^! S^\bullet \to 0\]
and for $i\geq 2$, we get isomorphisms
\[ \cohH^{d-i} i_x^! S^\bullet \cong H^{2d-i+1}_{\{x\}}(Z).\]

By \lemmaref{lem-CohomologyBounds} and \lemmaref{lem-SupportDimBound} we get (using that ${\rm HRH}_x(Z) \geq 0$):
\[ \cohH^{j} i_x^! S^\bullet = 0 \text{ for all } j \geq d -2{\rm HRH}_x(Z) - 2,\]
and so we get the desired vanishing.
\end{proof}

The cohomology of the link $L_x$ of $Z$ at $x$ can be described, following \cite{DurfeeSaito}, by the cohomology of the cone
\[ {\rm cone}( i_x^!\QQ^H_Z \to i_x^*\QQ^H_Z) \in D^b({\rm MHM}(\{x\})).\]

In other words, if $U = Z\setminus \{x\}$ with inclusion $j\colon U\to Z$, we have equality 
\[ H^k(L_x) = \cohH^k i_x^*j_*(\QQ^H_U).\]

Immediately from the definition, we get a long exact sequence of mixed Hodge structures \cite{DurfeeSaito}*{Prop. 3.5}:
\[ \dots \to H^k_{\{x\}}(Z) \to H^k(\{x\}) \to H^k(L_x) \to \dots.\]

As $\{x\}$ is simply a point, we get isomorphisms of mixed Hodge structures
\[ \QQ \cong H^0(L_x), \quad H^{1}_{\{x\}}(Z)=0, \quad H^k(L_x) \cong H^{k+1}_{\{x\}}(Z) \text{ for all } k \geq 1.\]

More generally, if $\mathcal S = \{S_\alpha\}_{\alpha \in I}$ is a Whitney stratification of $Z$, with $i_\alpha \colon S_\alpha \to Z$ the locally closed embedding, $j_\alpha \colon U_\alpha \to Z$ the inclusion of the complement, and $a_\alpha \colon S_\alpha \to *$ the constant map, \cite{DurfeeSaito}*{Prop. 2.11} explains that the cohomology $H^i(L_\alpha)$ of the link $L_\alpha$ of $Z$ at the stratum $S_\alpha$ is given by the cohomology
\[ \cohH^i (a_\alpha)_* i_\alpha^* j_{\alpha,*}(\QQ_{U_\alpha})\]
and hence we get the mixed Hodge structure on this cohomology by
\[ \cohH^i (a_\alpha)_* i_\alpha^* j_{\alpha,*}(\QQ_{U_\alpha}^H).\]

The following was pointed out to us by Lauren\c{t}iu Maxim.

\begin{lemma} \label{lem-LinkCoh} In the notation above, for any $\alpha \in I$ and $x\in S_\alpha$, there is an isomorphism of mixed Hodge structures
\[ H^i_{\{x\}}(Z) \cong H^{i-2\dim S_\alpha -1}(L_\alpha)(\dim S_\alpha).\]
\end{lemma}
\begin{proof} Let $T$ be a normal slice to $S_\alpha$ at $x$. To define this, near $x$ choose an embedding $i\colon Z \to X$ of $Z$ into a smooth algebraic variety $X$. Then there exists a submanifold $T$ of $X$ of dimension equal to ${\rm codim}_Z(S_\alpha)$ such that the intersection $T\cap S_\alpha = \{x\}$ is transversal.

The transversality implies that, locally, the inclusion map $i_T \colon T \to X$ is non-characteristic with respect to any mixed Hodge module which is constructible with respect to $\mathcal S$: indeed, Whitney's (a) condition implies $T$ has transversal intersection with any stratum $S_\beta$ such that $S_\alpha \subseteq \overline{S_\beta}$, and so if we remove the closures of strata which do not satisfy that property, we obtain this non-characteristic property. Note that removing the closures of those strata replaces $Z$ with an open neighborhood of $S_\alpha$, which does not matter due to the local nature of the claim.

Non-characteristic restrictions satisfy $(i_T)^! = (i_T)^*[-2\dim S_\alpha](\dim S_\alpha)$ (see, for example, \cite{SaitoMHM}*{Lem. 2.25}). Here we use that the codimension of the embedding $i_T$ is equal to $\dim S_\alpha$, by construction.

We have the following commutative diagram of embeddings
\[ \begin{tikzcd} &  Z \ar[dr,"i"] & \\ \{x\} \ar[r,"\kappa_x"] \ar[ru,"\iota_x"] \ar[swap, rd,"k_x"] & Z \cap T \ar[d] \ar[u] \ar[r]  &  X \\ &  T \ar[ru,swap,"i_T"] &  \end{tikzcd}\]
where $i \circ \iota_x = i_T \circ k_x = i_x \colon \{x\} \to X$.

Thus,
\[ \iota_x^! \QQ_Z^H = i_x^! i_* \QQ_Z^H = k_x^! i_T^!(i_* \QQ_Z^H) = k_x^! i_T^*(i_* \QQ_Z^H) [-2\dim S_\alpha](\dim S_\alpha)\]
\[ = \kappa_x^!(\QQ_{Z\cap T}^H) [-2\dim S_\alpha](\dim S_\alpha)\]
and we get the isomorphism of mixed Hodge structures
\[ H^i((i_x)^! \QQ_Z^H) \cong H^{i-2\dim S_\alpha}(\kappa_x^! \QQ^H_{Z\cap T})(\dim S_\alpha).\]

Finally, for $i-2\dim S_\alpha \geq 2$, we can identify the space on the right with 
\[ H^{i-2\dim S_\alpha -1}((Z\cap T) \setminus \{x\})(\dim S_\alpha)\] because we can take $Z\cap T$ to be contractible.

In the Euclidean topology, we have a local ``product neighborhood'' $U \cong (Z\cap T) \times B$ of the point $x$ in $Z$ where $B$ is a small ball centered at $x$ and with the property that $U\cap S_\alpha \cong \{x\}\times B$. The link $L_\alpha$ of $Z$ at $S_\alpha$ is homotopy equivalent to $U \setminus (U\cap S_\alpha) \cong (Z\cap T \setminus \{x\}) \times B$, hence to $Z\cap T \setminus \{x\}$ as $B$ is contractible. Thus, we have an isomorphism
\[ H^{i-2\dim S_\alpha-1}((Z\cap T)\setminus \{x\})(\dim S_\alpha) \cong H^{i-2\dim S_\alpha-1}(L_\alpha)(\dim S_\alpha),\]
proving the claim.
\end{proof}

\begin{theorem}\label{prop-LocCohAtPoint} Let $Z$ be a complex variety of pure dimension $d$ and let $x\in Z$ be a point. Then the following are equivalent for $k\in \ZZ_{\geq 0}$:
\begin{enumerate}
    \item $\HRH_x(Z)\geq k$,
    \item $F_{k-d} H^i_{\{y\}}(Z)= \begin{cases} 0 & i < 2d \\ \CC & i = 2d\end{cases}$ for all $y$ in a neighborhood of $x$ in $Z$,
    \item $F_{k-d} H^{i-1}(L_y)= \begin{cases} 0 & 0 < i < 2d \\ \CC & i = 2d\end{cases}$ for all $y$ in a neighborhood of $x$ in $Z$.
\end{enumerate} 
\end{theorem}
\begin{proof} As the claim is local, we can assume $Z$ is embeddable into a smooth variety. Let $i_x\colon \{x\} \to Z$ be the inclusion of a point $x$.

We have $\HRH_x(Z) \geq k$ iff in a neighborhood $U$ of $x$, the map
\[ F_{k-d}\psi_Z \colon F_{k-d}\QQ_Z^H[d] \to F_{k-d}(\mathbf D_Z(\QQ^H_Z[d])(-d))\]
is a quasi-isomorphism, which is equivalent to $F_{k-d}S^\bullet\vert_{U} = 0$, where as above $S^\bullet = {\rm cone}(\psi_Z)$.

By \lemmaref{lem-pInvariantVanishing}, we have $F_{k-d}S^\bullet\vert_U = 0$ if and only if $F_{k-d} i_y^!S^\bullet = 0$ for all $y\in U$, and the latter claim is true if and only if $F_{k-d}i_y^!\QQ^H_Z[d] \to F_k(\QQ^H[-d])$ is a quasi-isomorphism for all $y\in U$, giving the result.
\end{proof}

\begin{proof}[Proof of \theoremref{thm-characterize}\eqref{thm-link}] This follows immediately from \theoremref{prop-LocCohAtPoint} and \lemmaref{lem-LinkCoh}.
\end{proof}


We now specialize to the isolated singularities case. In \cite{FLIsolated}, the \emph{link invariants} of $Z$ at $x$ are defined by
\begin{equation} \label{eq-linkInvs} \ell^{p,q}= \dim_{\CC}{\rm Gr}_F^{p}H^{p+q}(L_x).\end{equation}

We can restate the condition $\HRH_x(Z)\geq k$ in terms of these invariants, and give a generalization of \cite{FLIsolated}*{Thm. 1.15(i)}. We first prove that, in the isolated singularities case, the only interesting cohomology for $L_x$ appears in a range governed by ${\rm lcdef}_x(Z)$.

\begin{lemma} \label{lem-LinkCohSuppIso} Let $x\in Z$ be an isolated non-rational homology manifold point in an irreducible variety with ${\rm lcdef}_x(Z) = a$ and $L_x$ the link of $Z$ at $x$. Then
\[ H^{i}_{\{x\}}(Z) \neq 0 \implies i \in \{2d\} \cup [d-a,d+a+1].\]
or equivalently, 
\[ H^i(L_x) \neq 0 \implies i \in \{0,2d-1\} \cup [d-a-1,d+a].\]
\end{lemma}
\begin{proof} The argument follows that of \lemmaref{lem-Irreducible} above. As $Z$ is irreducible, we know by that lemma that $a\leq d-2$.

We assume for now that $a> 0$. Recall that if $K_Z^\bullet$ is the RHM defect object (which by the assumption $a>0$ is non-zero), then $a = \max\{i \mid \cohH^{-i}K_Z^\bullet \neq 0\} = \max\{i \mid \cohH^i \mathbf D_Z(K_Z^\bullet) \neq 0\}$. We observed above that
\[ \cohH^{-i}K_Z^\bullet \cong \cohH^{-i-1}S^\bullet \text{ for all } i> 0,\]
\[ \cohH^{i}\mathbf D_Z(K_Z^\bullet)(-d) \cong \cohH^{i}S^\bullet \text{ for all } i\geq 0,\]
and so we get the implication:
\[ \cohH^i S^\bullet \neq 0 \implies i \in [-a-1,a].\]

This implication clearly holds true if $a = 0$, too, because then $S^\bullet$ is the cone of a morphism between two mixed Hodge modules on $Z$.

We have the triangle
\[ \QQ_Z^H[d] \to \mathbf D(\QQ_Z^H[d])(-d) \to S^\bullet \xrightarrow[]{+1}, \]
where by the isolated non-RHM locus assumption, $S^\bullet = i_* S'$ for some $S' \in D^b({\rm MHM}(\{x\}))$, with $i\colon \{x\} \to Z$ the embedding. By applying $i^!$, we get the triangle in $D^b({\rm MHS})$
\[ i^!\QQ_Z^H[d] \to \QQ^H[-d](-d) \to S' \xrightarrow[]{+1}.\]

Thus, for all $i \leq d-1$, we have isomorphisms
\[ \cohH^{i-1}S' \cong H^{d+i}_{\{x\}}(Z),\]
but note that $\cohH^{i-1}S' \neq 0 \implies i-1 \in [-a-1,a]$ as well. So we conclude that
\[ H^{d+i}_{\{x\}}(Z) \neq 0 \implies i \in [-a,a+1].\]

The claim for link cohomology follows from the isomorphism
\[ H^{i-1}(L_x) \cong H^i_{\{x\}}(Z) \text{ for all } i \geq 2.\]
\end{proof}

\begin{proof}[Proof of \corollaryref{corlinkiso}] The first statement trivially implies the second.

For the first statement, note that $\ell^{d-i,q-d+i} = 0$ for all $i\leq k$ if and only if $F^{d-i} H^{q}(L_x) = 0$ for all $i\leq k$ if and only if $F_{k-d}H^{q}(L_x) = 0$. So, by \theoremref{prop-LocCohAtPoint}, this shows that $\HRH_x(Z)\geq k$ implies those vanishings.

For the converse, we need to also check the following two statements:
\begin{enumerate}
    \item \label{itm-1} $F_{k-d}H^{p}(L_x) = 0$ for $p\in [d-a-1,d-1]$
    \item \label{itm-2} $F_{k-d} H^{2d-1}(L_x) = \CC$
\end{enumerate}

Note that Statement \eqref{itm-2} is automatic for $k=0$. Thus, after we prove that Statement \eqref{itm-1} holds for $k=0$, then we get Statement \eqref{itm-2} automatically for all $k$. Indeed, \theoremref{thm-TopLocCoh} implies in this case that $\dim_{\CC} H^{2d-1}(L_x) = 1$.

Thus, we have reduced to proving that Statement \eqref{itm-1} is true under the assumption $F_{k-d} H^p(L_x) = 0 \text{ for } p\in [d,d+a]$.

We apply \cite{FLIsolated}*{Prop. 2.8} and Serre duality \cite{FLIsolated}. The duality says that $\ell^{p,q} = \ell^{d-p,d-q-1}$. 

Thus, our assumption $\ell^{d-i,q-d+i} = 0$ gives $\ell^{i,2d-q-i-1} = 0$ for all $i\leq k$ and $q \in [d,d+a]$. In other words, for $p \in [d-a-1,d-1]$, we have $\ell^{i,p-i} = 0$ for all $i\leq k$. Then \cite{FLIsolated}*{Prop. 2.8} gives
\[ \sum_{i=0}^k \ell^{p-i,i} \leq \sum_{i=0}^k \ell^{i,p-i} = 0\]
and so $\ell^{p-i,i} = 0$ for all $i\leq k$, too. Thus,
\[ F_{k-d} H^p(L_x) \subseteq F_{k-p} H^p(L_x) = 0,\]
where the first containment uses that $p\leq d-1$. This completes the proof.
\end{proof}

\begin{remark} \label{rmk-RecoverFL} If $x\in Z$ is an isolated singular point and $Z$ has local complete intersection singularities near $x$, \cite{FLIsolated}*{Thm. 1.15(i)} says that $Z$ is $k$-rational near $x$ if and only if $Z$ is $k$-Du Bois near $x$ and $\ell^{k,d-k-1}=0$. 

In the local complete intersection setting, by the Serre duality relation $\ell^{p,q}= \ell^{d-p,d-q-1}$, this condition is equivalent to the vanishing $\ell^{d-k,k} = 0$.

In fact, this argument shows that for any $Z$ with an isolated singular point at $x$ and satisfying ${\rm lcdef}_x(Z) = 0$, the result of Friedman-Laza holds.
\end{remark}

\subsection{Behavior under finite group quotients and \'etale morphisms} Using the criteria for $\HRH(Z) \geq k$ in terms of $H^\bullet_{\{x\}}(Z)$, we see easily that $\HRH$ descends under finite group quotients and is preserved under \'{e}tale morphisms.

The following corollary should be compared to \cite{SVV}*{Prop. 4.2 (2)}.

\begin{corollary}\label{cor-descent2}
Let $\pi\colon Z\to W$ be the quotient of a variety $Z$ by the action of a finite group $G$. Then,
\[\HRH(W) \geq \HRH(Z).\]

In particular, if $Z$ and $W$ are in addition normal, and $Z$ has pre-$k$-rational singularities, then the singularities of $W$ are also pre-$k$-rational. 
\end{corollary}
\begin{proof}
Let $x\in Z$ such that $\pi(x)$ is singular. Using \cite{Brion}*{Proof of Prop. A1}, we have the isomorphisms of mixed Hodge structures \[H^i_{\left\{x\right\}}(Z)^{G_x}\cong H^i_{Gx}(Z)^G\cong H^i_{\left\{\pi(x)\right\}}(W).\]   
The assertion follows by taking Hodge pieces and applying \theoremref{prop-LocCohAtPoint}. 

The last assertion follows by combining this with \cite{SVV}*{Prop. 4.2 (1)}.
\end{proof}

\begin{lemma} Let $\varphi \colon Z_1 \to Z_2$ be a surjective \'{e}tale morphism. Then we have $\HRH(Z_1) = \HRH(Z_2)$.
\end{lemma}
\begin{proof} We have an isomorphism of mixed Hodge structures for any $x\in Z_1$
\[ H^i_{\{x\}}(Z_1) \cong H^i_{\{\varphi(x)\}}(Z_2),\]
and so the claim follows from \theoremref{prop-LocCohAtPoint}.
\end{proof}

By Luna's \'{e}tale slice theorem \cite{Luna}, the previous two lemmas show that the rational homology manifold condition descends under geometric quotients, which is probably well known to experts. Note that if the action does not give a geometric quotient, then the result fails (we learned this result from Wanchun Shen): the hypersurface $Z = V(x_1^2+ \dots +x_6^2) \subseteq \mathbb A^6$ is a GIT quotient which is not a rational homology manifold by \cite{DOR2}*{Eg. 7.7}.

\begin{corollary} Let $G$ be a reductive group acting on a smooth variety $X$ such that the quotient $Z = X/G$ is a geometric quotient. Then $Z$ is a rational homology manifold.
\end{corollary}

\begin{lemma} \label{lem-kPRSSmoothPullback} Let $\varphi \colon Z_1 \to Z_2$ be a smooth surjective morphism. Then $\HRH(Z_1) = \HRH(Z_2)$.
\end{lemma}
\begin{proof} As the question is local, we can reduce to the previous lemma and \lemmaref{lem-ProductPartialSmooth}. 
\end{proof}

In a similar vein, we have the following about pre-$k$-Du Bois singularities under smooth morphisms.

\begin{lemma} Let $\varphi \colon Z_1\to Z_2$ be a smooth surjective morphism. Then $Z_1$ is pre-$k$-Du Bois if and only if $Z_2$ is.
\end{lemma}
\begin{proof} The question is local, so we can prove the claim in two steps. First of all, if $\varphi$ is an \'{e}tale morphism, then $\underline{\Omega}_{Z_1}^p = \varphi^*(\underline{\Omega}_{Z_2}^p)$, and because $\varphi$ is faithfully flat, we see that the claim is true in this setting.

On the other hand, if $\varphi \colon Z_1 = Z_2 \times Y \to Z_2$ is a smooth projection, where $Y$ is a smooth variety, we want to understand $\underline{\Omega}_{Z_1}^p$ in terms of $\underline{\Omega}_{Z_2}^\bullet$ and $\underline{\Omega}_Y^\bullet = \Omega_Y^\bullet$. To do this, locally embed $i \colon Z_2 \subseteq W$ with $W$ a smooth variety. Let $\pi \colon W\times Y \to W$ be the smooth projection and let $i \times {\rm id}_Y = \iota \colon Z_2\times Y \to W\times Y$ be the closed embedding.

By applying the formula \eqref{eq-BoxTimesDecomposition} with $M^\bullet = i_* \QQ_{Z_2}^H$ and $N^\bullet = \QQ_Y^H$, we get
\[ \underline{\Omega}_{Z_2 \times Y}^k = \bigoplus_{i+j =k} \underline{\Omega}_{Z_2}^i \boxtimes \underline{\Omega}_{Y}^j = \bigoplus_{i+j=k} \underline{\Omega}_{Z_2}^i \boxtimes \Omega_Y^j,\]
where the latter equality is due to the smoothness of $Y$. This follows from the fact that \[\iota_* \QQ_{Z_2 \times Y}^H = \pi^*(i_* \QQ_{Z_2}^H) = i_* \QQ_{Z_2}^H \boxtimes \QQ_Y^H\]
and the equalities
\[ \underline{\Omega}_{Z_2 \times Y}^k = {\rm Gr}^F_{-k} {\rm DR}(\QQ_{Z_2 \times Y}^H)[k],\]
\[ \underline{\Omega}_{Z_2}^i = {\rm Gr}^F_{-i} {\rm DR}(\QQ_{Z_2}^H)[i],\]
\[ \underline{\Omega}_{Y}^j = {\rm Gr}^F_{-j} {\rm DR}(\QQ_{Y}^H)[j].\]

As $\Omega_Y^j$ is a sheaf, we see that $\underline{\Omega}_{Z_2\times Y}^\ell$ is a sheaf for all $\ell \leq k$ if and only if $\underline{\Omega}_{Z_2}^\ell$ is a sheaf for all $\ell \leq k$, proving the claim by definition of pre-$k$-Du Bois.
\end{proof}

We immediately obtain the following corollary.

\begin{corollary} Let $\varphi \colon Z_1 \to Z_2$ be a smooth surjective morphism between normal varieties. Then $Z_2$ is pre-$k$-rational if and only if $Z_1$ is pre-$k$-rational.
\end{corollary}
    
\subsection{Application to partial Poincar\'{e} duality}\label{sec-pd}
This short subsection proves a partial Poincar\'{e} duality result under the assumption $\HRH(Z)\geq k$. Our goal is to understand how the condition that $F_{k-d} \psi_Z$ is a quasi-isomorphism behaves under the direct image $(a_Z)_*$.

\begin{proposition} \label{prop-DirectImage} Let $f\colon X \to Y$ be a morphism of embeddable algebraic varieties. Let $\varphi \colon M^\bullet \to N^\bullet$ be such that $F_p \varphi$ is a quasi-isomorphism. Then $F_p f_*\varphi$ is a quasi-isomorphism.
\end{proposition}
\begin{proof} By \cite{Park}*{Lem. 3.4}, if $C^\bullet$ is the cone of $\varphi$, then our assumption implies
\[ {\rm Gr}^F_\ell{\rm DR}_Y(f_* C^\bullet) = 0 \text{ for all } \ell \leq p,\]
and so, by \lemmaref{lem-technical}, we get
\[ F_p f_* C^\bullet = 0.\]

Finally, as $f_*$ is an exact functor between triangulated categories, we have the exact triangle
\[ f_* M^\bullet \to f_* N^\bullet \to f_* C^\bullet \xrightarrow[]{+1},\]
and the result follows by looking at the long exact sequence in cohomology, using strictness of morphisms between Hodge modules.
\end{proof}

We can now prove the main result concerning Poincar\'{e} duality. We state the result using decreasing Hodge filtrations, as is the convention for the mixed Hodge structure on singular cohomology.

\begin{theorem}\label{thm-PD} Let $Z$ be an embeddable complex algebraic variety and assume $\HRH(Z)\geq k$. Then for any $i \in \ZZ$, the natural map
\[ H^{d-i}(Z) \to {\rm IH}^{d-i}(Z) \to (H^{d+i}_c(Z)^\vee)(-d)\]
induces isomorphisms
\[ F^{d-k} H^{d-i}(Z) \to F^{d-k} {\rm IH}^{d-i}(Z) \to F^{-k} H^{d+i}_c(Z)^{\vee}.\]

If $Z$ is proper with isolated non-rational homology manifold locus, then the converse holds.
\end{theorem}
\begin{proof} This follows by applying $\cohH^{-i} (a_Z)_*$ to the map $\psi_Z$, where $a_Z \colon Z \to {\rm pt}$ is the constant map. 

 After taking some embedding $i\colon Z \to X$ into a smooth variety, the condition ${\HRH}(Z) \geq k$ implies that the maps
\[ F_{k-d} \QQ_Z^H[d] \to F_{k-d} {\rm IC}_Z^H \to F_{k-d}\mathbf D_Z(\QQ_Z^H[d])(-d),\]
are quasi-isomorphisms. Applying $(a_Z)_*$ and using \propositionref{prop-DirectImage}, we get that
\[ F_{k-d}(a_Z)_*\QQ_Z^H[d] \to F_{k-d} (a_Z)_*{\rm IC}_Z^H \to F_{k-d}(a_Z)_*\mathbf D_Z(\QQ_Z^H[d])(-d)\]
are quasi-isomorphisms, too. By taking $\cohH^{-i}$, we get the desired claim.

For the converse, assume $Z$ is proper with isolated non-rational homology manifold locus. Then we have the exact triangle
\[ \QQ_Z^H[d] \to \mathbf D_Z(\QQ_Z^H[d])(-d) \to i_* S \xrightarrow[]{+1},\]
where $i\colon Z_{\rm nRS} \to Z$ is the inclusion of the singular locus. Note that ${\rm HRH}(Z) \geq k$ is equivalent to ${\rm Gr}^F_{-p} {\rm DR}(i_* S) = 0$ for all $p \leq k$. By properness of $i$, this is equivalent to $i_* {\rm Gr}^F_{-p}{\rm DR}(S) = 0$.

On the other hand, if we apply $(a_Z)_*$ to the exact triangle, we get
\[ (a_Z)_*\QQ_Z^H[d] \to (a_Z)_*\mathbf D_Z(\QQ_Z^H[d])(-d) \to S \xrightarrow[]{+1} \]
in $D^b({\rm MHS})$. By looking at the associated long exact sequence of mixed Hodge structures, and using strictness of the Hodge filtration, the isomorphism in the theorem statement is equivalent to the vanishing ${\rm Gr}^F_{-p} S = {\rm Gr}^F_{-p} {\rm DR}(S) = 0$ for all $p \leq k$, which proves the claim.
\end{proof}

\subsection{Generic local cohomological defect} \label{sec-GenLcdef} In this section, we provide a generalization of \remarkref{rmk-PPProperties}(4)(c) in the embedded setting. The key idea is that one can improve the important codimension bound on the non-rational homology manifold locus in \cite{PPLefschetz} by incorporating the ``generic'' local cohomological defect, which we introduce next.

Let $Z\subseteq X$ be an embedding of the purely $d$-dimensional variety $Z$ into a smooth connected variety $X$.

Recall that $K_Z^\bullet$ is the RHM defect object, which lies in the exact triangle
\[ K_Z^\bullet \to \QQ_Z^H[d] \to {\rm IC}_Z^H \xrightarrow[]{+1}.\]

Let $\mathcal S =\{S_\alpha\}_{\alpha \in I}$ be any Whitney stratification of $X$ such that $Z$ is a union of strata. Note that $K_Z^\bullet$ is constructible with respect to any Whitney stratification of $Z$, hence, with respect to $\mathcal S$.

In particular, the perverse cohomology sheaves of $\QQ_Z[d]$ have support equal to some unions of strata in $\mathcal S$, and so the function
\[ x \mapsto {\rm lcdef}_x(Z)\]
is constructible with respect to this stratification.

\begin{definition}\label{deflcdefgen}
Let $\mathcal S =\{S_\alpha\}_{\alpha \in I}$ be a Whitney stratification as above and 
${\rm lcdef}_{S_\alpha}(Z)$
the value of ${\rm lcdef}$ on $S_\alpha$. We denote
\[{\rm lcdef}_{\rm gen}(Z) = \max\bigl(\{0\}\cup\{{\rm lcdef}_{S_\alpha}(Z) \mid \dim S_\alpha = \dim Z_{\rm nRS}\}\bigr).\]
\end{definition}

This invariant can be detected through local cohomology in our embedded setting.

\begin{lemma} Let $Z \subseteq X$ be an embedding into a smooth variety $X$ with ${\rm codim}_X(Z) = c$. Let $\{S_\alpha\}$ be a Whitney stratification of $X$ as above. Then
\[ {\rm lcdef}_{S_\alpha}(Z) = \max\bigl(\{0\} \cup \{ i \mid S_\alpha \subseteq {\rm Supp} \, \cohH^{c+i}_Z(\shO_X)\}\bigr).\]
\end{lemma}
\begin{proof} This is immediate by choice of stratification.
\end{proof}

This invariant can also be detected without mention of a Whitney stratification.

\begin{lemma}\label{lcdefgen} Let $Z \subseteq X$ be an embedding into a smooth variety $X$, with ${\rm codim}_X(Z) = c$. For $i\geq 0$, define
\[ d(i) = \begin{cases} \dim {\rm Supp} \, \cohH^{c+i}_Z(\shO_X) & i > 0 \\ \dim {\rm Supp}\, \cohH^c_Z(\shO_X)/{\rm IC}_Z^H & i = 0 \end{cases}.\]
Then \[{\rm lcdef}_{\rm gen}(Z) = \max\bigl(\{0\} \cup \{ i \mid d(i) \geq d(j) \text{ for all } j\geq 0\}\bigr).\]
\end{lemma}
\begin{proof} By definition, we have
\[ Z_{\rm nRS} = {\rm Supp}(K_Z^\bullet) = {\rm Supp}(\mathbf D_Z(K_Z^\bullet)) = {\rm Supp}(\cohH^c_Z(\shO_X)/{\rm IC}_Z^H)\cup \bigcup_{i > 0} {\rm Supp} \cohH^{c+i}_Z(\shO_X).\]

Note that, for any $i$ which satisfies $d(i) \geq d(j)$ for all $j\geq 0$, we have $d(i) = \dim Z_{\rm nRS}$. Thus, the claim is immediate by the previous lemma and the fact that the support of $Z_{\rm nRS}$ is a union of strata.
\end{proof}

Now, we show that the value ${\rm lcdef}_x(Z)$ for $x\in S_\alpha$ is unchanged upon taking a normal slice.

\begin{proposition} \label{normalslicelcdef} Let $\{S_\alpha\}$ be a Whitney stratification of $X$ as above and fix $x\in S_\alpha$. Let $T_\alpha \subseteq X$ be a normal slice through $x$, meaning a smooth subvariety of dimension $\dim X - \dim S_\alpha$ such that $T_\alpha \cap S_\alpha = \{x\}$ is a transverse intersection. Then
\[{\rm lcdef}_x(Z) = {\rm lcdef}_x(Z\cap T_\alpha).\]
\end{proposition}
\begin{proof} By replacing $X$ with a neighborhood of $x$, we can assume the stratification $\{S_\alpha\}$ is finite. By further replacing $X$ by $X \setminus \bigcup_{\beta \in B} \overline{S_\beta}$, where $B = \{\gamma \mid S_\alpha \not \subseteq \overline{S_\gamma}\}$, we can assume $S_\alpha$ is a minimal stratum in the sense that $S_\alpha \subseteq \overline{S_\gamma}$ for all $\gamma$. As the support of each local cohomology is a union of strata and closed, we see then that after this restriction, we have
\[ {\rm lcdef}_x(Z) = {\rm lcdef}_{S_\alpha}(Z) = {\rm lcdef}(Z).\]

In this case, Whitney's condition (a) implies that the normal slice $T_\alpha$ has transverse intersection with all strata. If $\iota \colon T_\alpha \to X$ is the closed embedding, this implies that $\iota$ is non-characteristic with respect to $\cohH^j(K_Z^\bullet)$ for all $j\in \ZZ$. 

Thus, the spectral sequence
\[ E_2^{i,j} = \cohH^i \iota^* \cohH^j(K_Z^\bullet) \implies \cohH^{i+j} \iota^*(K_Z^\bullet)\]
degenerates at $E_2$, because the only non-zero terms must have $i = -\dim S_\alpha$, the codimension of the embedding $\iota$. This gives equality
\[ \cohH^{j-\dim S_\alpha} \iota^*(K_Z^\bullet) = \cohH^j K_Z^\bullet \otimes_{\shO_X} \shO_{T_\alpha}.\]

Again using that $\iota$ is non-characteristic, we have
\[ \iota^*(K_Z^\bullet) = K_{Z\cap T_\alpha}^\bullet[\dim S_\alpha].\]

Finally, using that ${\rm lcdef}(Z) = \max\{i \mid \cohH^{-i} K_Z^\bullet \neq 0\}$ and the same formula for ${\rm lcdef}(Z\cap T_\alpha)$, we get the desired equality.
\end{proof}

Putting these together, we can show that the bound ${\rm codim}_Z(Z_{\rm nRS}) \geq 2\HRH(Z) +3$ can be improved if one knows the ``generic'' local cohomological defect. This method of proof is inspired by \cite{SaitoMicrolocal}*{Rmk. 2.11}.

\begin{proposition}\label{prop-ppbound} Let $Z$ be a purely $d$-dimensional complex algebraic variety and let $\{S_\alpha\}$ be a Whitney stratification of $Z$. Then for all $x\in S_\alpha$, we have
\[ {\rm lcdef}_{S_\alpha}(Z) \leq \max\{{\rm codim}_Z(S_\alpha) - 2 \HRH_x(Z) -3 , 0\}.\]
In particular, we have
\[ {\rm lcdef}_{\rm gen}(Z) \leq \max\{{\rm codim}_Z(Z_{\rm nRS})-2\HRH(Z) -3,0\}.\]
\end{proposition}
\begin{proof} The first claim implies the second by first taking $\alpha$ such that ${\rm lcdef}_{S_\alpha}(Z) = {\rm lcdef}_{\rm gen}(Z)$ and $\dim S_\alpha = \dim Z_{\rm nRS}$, and then noting that $\HRH(Z) \leq \HRH_x(Z)$. So it suffices to prove the first claim. 

For $x\in S_\alpha$, take a normal slice $T_\alpha$ through $x$. The first claim is then immediate from the fact that $\HRH_x(Z) \leq \HRH_x(Z\cap T_\alpha)$ and the inequality (from \remarkref{rmk-PPProperties}(4d))
\[ {\rm lcdef}_{S_\alpha}(Z)= {\rm lcdef}_x(Z) = {\rm lcdef}_x(Z\cap T_\alpha) \leq \max\{\dim(Z\cap T_\alpha) - 2\HRH_x(Z\cap T_\alpha)-3,0\}.\]
This completes the proof.

As an alternative proof, one can use \lemmaref{lem-SupportDimBound}. Indeed, assuming $\HRH(Z) \geq 0$ this lemma tells us that
\[ \dim {\rm Supp}\cohH^i \mathbf D_Z(K_Z^\bullet) \leq \dim(Z) - 2\HRH(Z) -3-i,\]
and so choosing $i$ maximal so that $S_\alpha \subseteq {\rm Supp}(\cohH^i \mathbf D_Z(K_Z^\bullet))$, we get
\[ \dim S_\alpha \leq \dim {\rm Supp} \cohH^i \mathbf D_Z(K_Z^\bullet) \leq \dim(Z) - 2\HRH(Z) -3 -i,\]
but by definition, such an $i$ is ${\rm lcdef}_{S_\alpha}(Z)$.
\end{proof}

\begin{remark}
When the locus $Z_{\textrm{nRS}}$ is isolated, we have ${\rm lcdef}_{\rm gen}(Z)={\rm lcdef}(Z)$ (in particular, if $\dim Z\leq 3$ and $\HRH(Z)\geq 0$, we always have equality). However, equality can also hold even when $Z_{\textrm{nRS}}$ is non-isolated, see the examples in \S \ref{sec-determinantal}.
\end{remark}

\section{Examples}\label{sect-Examples}

This section is devoted to providing various examples with different features.

\subsection{Affine cones, toric and secant varieties} We first calculate the HRH level of affine cones over smooth projective varieties following the treatment of \cites{SVV}.

\begin{proposition}\label{prop-HRHCones}
Let $X$ be a smooth projective variety of dimension $n$, and $L$ be an ample line bundle on $X$. Let \[Z=C(X,L):=\mathrm{Spec}\left(\bigoplus_{m\geq 0}H^0(X,L^m)\right)\] be the affine cone with conormal bundle $L$. Then, for $k\leq \frac{n+1}{2}$, we have $\HRH(Z) \geq k$ if and only if the following two conditions are satisfied:
\begin{itemize}
    \item[(1)] $H^i(\Omega_X^p)=0$ for all $i\neq p, 0\leq p\leq k$,
    \item[(2)] $H^0(\mathcal{O}_X)\xrightarrow{\cup c_1(L)} H^1(\Omega_X^1)\xrightarrow{\cup c_1(L)}\cdots\xrightarrow{\cup c_1(L)}H^k(\Omega_X^k)$ are all isomorphisms. 
\end{itemize}
\end{proposition}
\begin{proof}
The blow up $f\colon \widetilde{Z}\to Z$ at the cone point $v$ is a strong log resolution of $Z$ with $E\cong X$. Note that by our assumption, $k<\textrm{codim}_Z(Z_{\textrm{sing}})=n+1$, whence the bottom map in  \eqref{pic1} is an isomorphism by \cite{SVV}*{Lem. 2.4}. In what follows, we use  \eqref{pic1} without any further reference. 

Let us first prove the assertion when $k=0$. We have the distinguished triangle \[\underline{\Omega}_Z^0\to {R}f_*\mathcal{O}_{\widetilde{Z}}\oplus\mathcal{O}_v\to { R}f_*\mathcal{O}_X \xrightarrow{+1}.\]
Since $Z$ is affine, using the arguments of \cite{SVV}*{Appendix A.1}, we see that the above induces 
\begin{equation}\label{eq2}
    0\to\Gamma(\mathcal{H}^0(\underline{\Omega}_Z^0))\to H^0(\mathcal{O}_{\widetilde{Z}})\oplus\mathbb{C}\to H^0(\mathcal{O}_X)\to 0,
\end{equation}
\begin{equation}\label{eq3}
    0\to\Gamma(\mathcal{H}^i(\underline{\Omega}_Z^0))\to H^i(\mathcal{O}_{\widetilde{Z}})\to H^i(\mathcal{O}_X)\to 0\quad \forall i\geq 1.
\end{equation}
By {\it loc. cit.}, $H^0(\mathcal{O}_{\widetilde{Z}})=\bigoplus_{m\geq 0}H^0(L^m)$ and the map $H^0(\mathcal{O}_{\widetilde{Z}})\oplus\mathbb{C}\to H^0(\mathcal{O}_X)$ in \eqref{eq2} sends $(x,\alpha)\mapsto \varphi(x)-\alpha$ where $\varphi\colon \bigoplus_{m\geq 0}H^0(L^m)\to H^0(\mathcal{O}_X)$ is the projection. Thus the composed map \[\Gamma(\mathcal{H}^0(\underline{\Omega}_Z^0))\to H^0(\mathcal{O}_{\widetilde{Z}})\oplus\mathbb{C}\to H^0(\mathcal{O}_{\widetilde{Z}})\] is an isomorphism. Finally, by \eqref{eq3}, for $i\geq 1$, $\Gamma(\mathcal{H}^i(\underline{\Omega}_Z^0))\to H^i(\mathcal{O}_{\widetilde{Z}})$ is an isomorphism if and only if $H^i(\mathcal{O}_X)=0$, which concludes the proof for $k=0$.

We use induction on $k$. Assume $k\geq 1$ and note the distinguished triangle induced by the residue sequence on $\widetilde{Z}$:
\begin{equation}\label{res}
    {R}f_*\Omega_{\widetilde{Z}}^k\to{ R}f_*\Omega_{\widetilde{Z}}^k(\log X)\to { R}f_*\Omega_X^{k-1}\xrightarrow{+1}.
\end{equation}

It is shown in \cite{SVV}*{Appendix A.2} that 
\begin{equation}\label{eq4}
    H^i(\Omega_{\widetilde{Z}}^k)=\bigoplus_{m\geq 0}H^i(\Omega_X^k\otimes L^m)\oplus \bigoplus_{m\geq 1}H^i(\Omega_X^{k-1}\otimes L^m)
\end{equation}
and the map
\begin{equation}\label{gamma}
    \Gamma(\mathcal{H}^i(\underline{\Omega}_Z^k))\to H^i(\Omega_{\widetilde{Z}}^k)
\end{equation}
induced from 
\[\underline{\Omega}_Z^k\to { R}f_*\Omega_{\widetilde{Z}}^k\to {R}f_*\Omega_X^{k}\xrightarrow{+1}\]
realizes $\Gamma(\mathcal{H}^i(\underline{\Omega}_Z^k))$ as the following direct summand of $H^i(\Omega_{\widetilde{Z}}^k)$ through \eqref{eq4}:
\[\Gamma(\mathcal{H}^i(\underline{\Omega}_Z^k))=\bigoplus_{m\geq 1}H^i(\Omega_X^k\otimes L^m)\oplus \bigoplus_{m\geq 1}H^i(\Omega_X^{k-1}\otimes L^m).\]
Thus, we are reduced to showing that the composition 
\begin{equation}\label{comp}
    \Gamma(\mathcal{H}^i(\underline{\Omega}_Z^k))\hookrightarrow H^i(\Omega_{\widetilde{Z}}^k)\to H^i(\Omega_{\widetilde{Z}}^k(\log X))
\end{equation}
(the first map is \eqref{gamma} and the second one is induced by \eqref{res})
is an isomorphism for all $i$ if and only if the two conditions in the statement are satisfied. If $i\neq k-1,k$, then we see from our induction hypothesis and \eqref{res} that the second map above is an isomorphism, whence the composition is an isomorphism if and only if $H^i(\Omega_X^k)=0$. Note that the connecting map 
\[H^{i}(\Omega_X^{k-1})\to H^{i+1}(\Omega_{\widetilde{Z}}^k)\] arising from \eqref{res} is the cup product $\cup c_1(L)$ map and lands in the direct summand $H^{i+1}(\Omega_X^k)$. This is injective for $i=k-1$ by Hard Lefschetz, by our assumption on $k$. Thus the second map in \eqref{comp} is an isomorphism for $i=k-1$, whence the composition is an isomorphism if and only if $H^{k-1}(\Omega_X^k)=0$. Finally, by the above argument, \eqref{res} induces the exact sequence 
\[ 0\to H^{k-1}(\Omega_X^{k-1})\to H^k(\Omega_{\widetilde{Z}}^k)\to H^k(\Omega_{\widetilde{Z}}^k(\log X))\to 0.\] It follows immediately from the description of $\Gamma(\mathcal{H}^i(\underline{\Omega}_Z^k))$ as a direct summand of $H^i(\Omega_{\widetilde{Z}}^k)$ that the composed map \eqref{comp} is an isomorphism for $i=k$ if and only if \[H^{k-1}(\Omega_X^{k-1})\xrightarrow{\cup c_1(L)}H^k(\Omega_X^k)\] is an isomorphism. That completes the proof.
\end{proof}

\begin{remark}
As before, let $X$ be a smooth projective variety of dimension $n$, $L$ be an ample line bundle on $X$, and $Z=C(X,L)$. Since $Z$ is singular only at the cone point, we have \[{\rm lcdef}_{\rm gen}(Z)={\rm lcdef}(Z)\]
\[=\min
\left\{
c\in\mathbb{N} \biggm| \begin{array}{l}
H^i(X,\mathbb{C})\xrightarrow{\cup c_1(L)} H^{i+2}(X,\mathbb{C}) \textrm{ are isomorphisms for all} -1\leq i\leq n-3-c,\\
\textrm{ and injective for }i=n-2-c\textrm{ with the convention that } H^{-1}(X,\CC)=0
\end{array}
\right\}
\]
where the last equality comes from \cite{PS}*{Thm. 6.1}. In fact, when $X\subseteq\mathbb{P}^N$, setting $Z=C(X)\subseteq\mathbb{A}^{N+1}$ to be the affine cone, we have the following by \cite{PS}*{Thm. A}
\[{\rm lcdef}_{\rm gen}(Z)={\rm lcdef}(Z)=\min\left\{c\in\mathbb{Z}_{\geq 0} \mid H^i(\mathbb{P}^N,\mathbb{C})\xrightarrow{\sim} H^i(X,\mathbb{C})\,\forall\, i\leq n-1-c\right\}.\]
\end{remark}

\begin{example}\label{ex-hrh-vs-db} Here are two explicit examples of affine cones with interesting properties:

(1) There are varieties with $k$-Du Bois singularities that satisfy $\HRH(Z) \geq k$ but their singularities are not $k$-rational (in the sense of \cite{SVV}). For example, consider $Z=C(\mathbb{P}^2,\mathcal{O}_{\mathbb{P}^2}(2))$. Then $Z$ is a rational homology manifold by \propositionref{prop-HRHCones}, but according to \cite{SVV}*{Prop. F}, its singularities are $1$-Du Bois but not $1$-rational.

(2) Let $X$ be a Godeaux surface, in other words, a surface satisfying $p_g=q=0$, $h^{1,1}(X)=9$ and having ample canonical bundle $\omega_X$. Thus, the cone $Z=C(X,\omega_X)$ satisfies $\HRH(Z)=0$ by the above proposition. But the singularities of $Z$ are not pre-$0$-Du Bois by \cite{SVV}*{Prop. F} as $h^2(\omega_X)\neq 0$. This is an example of a variety with $\HRH(Z)$ strictly higher than its Du Bois level.
\end{example}

\begin{example}[Normal affine toric varieties]
    Let $Z$ be a normal affine toric variety. By \cite{SVV}*{Prop. E} and \remarkref{rmk-PPProperties}, we obtain
    \[\HRH(Z)= \begin{cases} 
      +\infty & \textrm{if $Z$ is simplicial,} \\
      0 & \textrm{otherwise}.
   \end{cases}
\]
\end{example}

\begin{example} The recent preprints \cites{KV1,KV2} compute many invariants concerning the trivial Hodge module $\QQ_Z^H[\dim Z]$ for $Z$ a toric variety. Using this result, those authors provide an example with ${\rm lcdef}_{\rm gen}(Z) < {\rm lcdef}(Z)$ of dimension 4, which is the minimal dimension where such behavior can occur. In fact, their example \cite{KV2}*{Eg. 4.3} is such that the inequality of \theoremref{thm-ppbound} does not hold with ${\rm lcdef}(Z)$ in place of ${\rm lcdef}_{\rm gen}(Z)$. 
\end{example}

\begin{example}[Secant varieties]
    Let $X\subset \mathbb{P}^N$ be a smooth projective variety embedded by the complete linear series of a sufficiently positive line bundle (i.e. one that satisfies $(Q_1)$-property in the sense of
 \cite{SecantVarieties}*{Def. 3.1}). Let $\Sigma$ be its secant variety. Then by \cite{SecantVarieties} and \remarkref{rmk-PPProperties},  
\[\HRH(\Sigma)= \begin{cases} 
      +\infty & \textrm{if $X\cong\mathbb{P}^1$,} \\
      0 & \textrm{if $H^i(\mathcal{O}_X)=0$ for all $i\geq 1$ and $X\ncong\mathbb{P}^1$,} \\
      -1 & \textrm{otherwise}.
   \end{cases}
\]

This also follows from \cite{CDOR}*{Cor. B}.
\end{example}

\subsection{Ideals of generic, symmetric and skew-symmetric minors} \label{sec-determinantal}
In this subsection, we show how the main theorem of \cite{RW} can be used to compute ${\rm lcdef}_{\rm gen}(-)$ and to give a bound on ${\rm HRH}(-)$ for determinantal varieties.

We will be interested in subvarieties defined by matrices of appropriate ranks of the following spaces: \begin{enumerate} \item (Generic) $X = {\rm Mat}_{m,n}(\CC)$ with $m\geq n$, \item (Odd skew) $X = {\rm Mat}_{n}(\CC)^{\rm skew}$, $n$ odd, \item  (Even skew) $X = {\rm Mat}_{n}(\CC)^{\rm skew}$, $n$ even, \item (Symmetric) $X = {\rm Mat}_{n}(\CC)^{\rm sym}$. \end{enumerate}

In cases (1) and (4), we let $Z_p$ denote the subvariety of matrices of rank $\leq p$ and in cases (2) and (3), we let $Z_p$ denote the subvariety of matrices of rank $\leq 2p$. In every case, $Z_p$ is known to have rational singularities by \cite{Boutot}. Thus, ${\HRH}(Z_p) \geq 0$.

Following \cite{RW}, we let $D_p$ be the intersection homology $\Dmod_X$-module associated to the trivial local system on $Z_{p,\rm reg}$. We let $\Gamma(X)$ denote the Grothendieck group of holonomic $\Dmod_X$-modules. For $p$ fixed, we write
\[ H_p(q) = \sum_{j\geq 0} \left[\cohH_{Z_p}^j(\shO_X)\right]\cdot q^j \in \Gamma(X)[q].\]

Finally, for $a\geq b\geq 0$, let $\binom{a}{b}_q$ be the $q$-binomial coefficient, defined by
\[ \binom{a}{b}_q = \frac{(1-q^a)\dots (1-q^{a-b+1})}{(1-q^b)\dots (1-q)}.\]

The following computation will be important below:

\begin{lemma} \label{lem-qbinomial} Let $a \geq b \geq 0$. Then
\[ \binom{a}{b}_{q^{-4}} = q^{-4b(a-b)} + \text{ higher order terms}.\]
\end{lemma}
\begin{proof} We can write
\[ \binom{a}{b}_{q^{-4}} = \frac{(1-q^{-4a})\dots (1-q^{-4(a-b+1)})}{(1-q^{-4b})\dots (1-q^{-4})}\]
\[ = \frac{q^{-4(\sum_{i=1}^b (a-b+i))}}{q^{-4(\sum_{i=1}^b i)}}\cdot \frac{(q^{4a}-1)\dots (q^{4(a-b+1)} -1)}{(q^{4b}-1)\dots (q^{4}-1)}\]
\[ = q^{-4b(a-b)}(1 + \text{ higher order terms})\]
which completes the proof.
\end{proof}

Now, we can state the main result of \cite{RW}, which gives a formula for $H_p(q) \in \Gamma(X)[q]$.

\begin{theorem}[\cite{RW}*{Main Thm.}]\label{thm-rw} In the notation above, we have the following formula for $H_p(q)$ in the cases (1)-(4).
\begin{enumerate} \item (Generic) For all $0\leq p < n$, we have \[H_p(q) = \sum_{s=0}^p [D_s]\cdot q^{(n-p)^2+(n-s)(m-n)} \binom{n-s-1}{p-s}_{q^2}.\]
\item (Odd skew) Write $n= 2m+1$, then for all $0\leq p< m$, we have \[H_p(q) = \sum_{s=0}^p [D_s]\cdot q^{2(m-p)^2+(m-p) +2(p-s)} \binom{m-1-s}{p-s}_{q^4}.\]
\item (Even skew) Write $n = 2m$, then for all $0\leq p < m$, we have \[H_p(q) = \sum_{s=0}^p [D_s]\cdot q^{2(m-p)^2-(m-p)} \binom{m-1-s}{p-s}_{q^4}.\]
\item (Symmetric) For all $0\leq p < n$, we have \[H_p(q) = \sum_{\ell = 0}^{\lfloor \frac{p}{2}\rfloor} [D_{p-2\ell}]\cdot q^{1+\binom{n-p+2\ell+1}{2} - \binom{2\ell+2}{2}}\binom{\lfloor \frac{n-p+2\ell-1}{2}\rfloor}{\ell}_{q^{-4}}.\]
\end{enumerate}
\end{theorem}

As $\min\{j \mid \cohH^j_{Z_p}(\shO_X) \neq 0\} = {\rm codim}_X(Z_p)$, these expressions immediately imply the following well-known formulas for the (co)dimension of $Z_p$ in $X$ (see \cite{BrunsVetter} and \cite{Weyman}).

\begin{corollary}\label{cor-dim} In the notation above, we have the following formulas in cases (1)-(4).
\begin{enumerate} \item (Generic) In this case, we have \[{\rm codim}_{X}(Z_p) =(m-p)(n-p),\textrm{ and }
\dim(Z_p) = p(m+n-p).\]

\item (Odd skew) In this case, we have \[{\rm codim}_{X}(Z_p) =(m-p)(2(m-p)+1), \textrm{ and }\dim(Z_p) = p(2(n-p)-1).\]

\item (Even skew) In this case, we have \[{\rm codim}_{X}(Z_p) =(m-p)(2(m-p)-1), \textrm{ and }\dim(Z_p) = p(2(n-p)-1).\]

\item (Symmetric) In this case, we have \[{\rm codim}_{X}(Z_p) =\binom{n-p+1}{2}, \textrm{ and }\dim(Z_p) = \frac{1}{2}p(2n-p+1).\]
\end{enumerate}
\end{corollary}

We ignore $p=0$ as in this case $Z_p$ is smooth, and everywhere below we assume $p\geq 1$. In fact, in case (4), when $p=1$, \theoremref{thm-rw} implies $Z_1$ is a rational homology manifold (it is known that $\textrm{lcdef}(Z_1)=0$, and by the argument of \propositionref{hc=ic} below, ${\rm IC}_{Z_1} = \cohH^{{\rm codim}_X(Z_1)}_{Z_1}(\shO_X)$ which by \theoremref{thm-characterize}\eqref{thmmain} implies $Z_1$ is a $\mathbb{Q}$-homology manifold) so we will assume in case (4) that $p \geq 2$. More precisely, in what follows, our assumptions are as follows:
 \begin{itemize}
     \item In case (1), $1\leq p<n$.
     \item In cases (2), (3), $1\leq p<m$.
     \item In case (4), $2\leq p<n$.
 \end{itemize}

By definition, $D_p = {\rm IC}_{Z_p}$, and so we can study for which $Z_p$ we have equality ${\rm IC}_{Z_p} = \cohH^{{\rm codim}_X(Z_p)}_{Z_p}(\shO_X)$.

\begin{proposition}\label{hc=ic} In the notation above, let $c_p = {\rm codim}_X(Z_p)$. 
\begin{enumerate} \item (Generic) We have equality ${\rm IC}_{Z_p} = \cohH^{c_p}_{Z_p}(\shO_X)$ if and only if $m > n$.

\item (Odd skew) We have equality ${\rm IC}_{Z_p} = \cohH^{c_p}_{Z_p}(\shO_X)$ for every $1 \leq p < m$.

\item (Even skew) We have inequality ${\rm IC}_{Z_p} \neq \cohH^{c_p}_{Z_p}(\shO_X)$ for every $1 \leq p < m$.

\item (Symmetric) We have equality ${\rm IC}_{Z_p} = \cohH^{c_p}_{Z_p}(\shO_X)$ if and only if $n \equiv p \mod 2$ (when $p\geq 2$). 
\end{enumerate}
\end{proposition}
\begin{proof} By definition, one needs only check for which $s < p$ the summand $[D_s]$ appears as a coefficient of $q^{c_p}$ using \theoremref{thm-rw}. The first three claims are immediate, using that the lowest degree term of the $q$-binomial coefficient in each expression is the constant term. Note that in case (1), when $m =n$, every $[D_s]$ appears as a coefficient of $q^{c_p} = q^{(m-p)(n-p)}$.

For the case of symmetric matrices, we look at the lowest degree term in
\[q^{1+\binom{n-p+2\ell+1}{2} - \binom{2\ell+2}{2}}\binom{\lfloor \frac{n-p+2\ell-1}{2}\rfloor}{\ell}_{q^{-4}},\]
which by \lemmaref{lem-qbinomial}, is at degree
\begin{equation}\label{degree}
    1+ \binom{n-p+2\ell+1}{2} - \binom{2\ell+2}{2} -4\ell\left(\bigg\lfloor \frac{n-p+2\ell-1}{2}\bigg\rfloor - \ell\right).
\end{equation} 

We break into two cases depending on the class of $n-p$ mod $2$. We write the difference as \[n - p = \begin{cases} 2k \\ 2k+1 \end{cases},\]
so that the expression in \eqref{degree} can simplify into
\[ \begin{cases} 1+ \binom{2k+2\ell+1}{2} - \binom{2\ell+2}{2} -4\ell(k-1) & \textrm{ when }n-p=2k\\ 1+ \binom{2k+2\ell+2}{2} - \binom{2\ell+2}{2} -4\ell k & \textrm{ when }n-p=2k+1\end{cases},\]
and these are easily seen to simplify to
\[ \begin{cases} 2\ell + k(2k+1) & \textrm{ when }n-p=2k\\ (k+1)(2k+1) & \textrm{ when }n-p=2k+1\end{cases}.\]
This proves the claim, as the expression in the second case does not depend on $\ell$, meaning all $[D_{p-2\ell}]$ appear in $\cohH^{c_p}_{Z_p}(\shO_X)$, and the first case is minimized for $\ell = 0$.
\end{proof}

Now we can compute ${\rm lcdef}_{\rm gen}(Z_p)$. To do this, in cases (1)-(3), we look for the highest index such that $[D_{p-1}]$ appears with non-zero coefficient. In case (4), we look for the highest index with $[D_{p-2}]$ having non-zero coefficient. To get the defect, we subtract the codimension.

\begin{proposition}\label{lcdefgendet} In the notation above, we have the following formulas for ${\rm lcdef}_{\rm gen}(Z_p)$.
\begin{enumerate} \item (Generic) ${\rm lcdef}_{\rm gen}(Z_p) = m+n-2p-2$.

\item (Odd skew) ${\rm lcdef}_{\rm gen}(Z_p) = 4(m-p-1)+2$.

\item (Even skew) ${\rm lcdef}_{\rm gen}(Z_p) = 4(m-p-1)$.

\item (Symmetric) ${\rm lcdef}_{\rm gen}(Z_p) = 2(n-p-1)$ (we assume $p\geq 2$).
\end{enumerate}
\end{proposition}
\begin{proof} In the first three cases, we take $s = p-1$ and use that
\[ (\text{Generic}): \binom{n-s-1}{1}_{q^2} = \frac{(1-q^{2(n-p)})}{(1-q^2)} = 1+q^2 + \dots + q^{2(n-p-1)}.\]
\[(\text{Skew}):\binom{m-s-1}{1}_{q^4} = \frac{(1-q^{4(m-p)})}{(1-q^4)} =1+ q^4 + \dots + q^{4(m-p-1)}.\]

In the symmetric case, we take $\ell = 1$ and use that the maximal degree term in the $q$-binomial coefficient (evaluated at $q^{-4}$) is the constant term, so the $q$-binomial coefficient can be ignored in this case.

In summary, the highest degree with $[D_{p-1}]$ (in cases (1)-(3)) or $[D_{p-2}]$ (in case (4)) is
\begin{itemize} \item (Generic) $(n-p)^2 + (n-p+1)(m-n) + 2(n-p-1)$,

\item (Odd skew) $2(m-p)^2 +(m-p)+2 + 4(m-p-1)$,

\item (Even skew) $2(m-p)^2 -(m-p) + 4(m-p-1)$,

\item (Symmetric) $1+\binom{n-p+3}{2} -\binom{4}{2} = \binom{n-p+3}{2} - 5$.
\end{itemize}

The claim then follows by subtracting ${\rm codim}_{X}(Z_p)$ from each of these expressions.
\end{proof}

In \cite{RW}, a formula for the local cohomological dimension 
\[{\rm lcd}(X,Z_p) = \max\{j \mid \cohH^j_{Z_p}(\shO_X) \neq 0\}\]
is given. Using this and the computation at the end of the previous proof, we can compute the difference ${\rm lcdef}(Z_p) - {\rm lcdef}_{\rm gen}(Z_p)$.

\begin{proposition} \label{prop-lcdefDifference} In the notation above, we have the following formulas. 

\begin{enumerate}\item (Generic) ${\rm lcdef}(Z_p) - {\rm lcdef}_{\rm gen}(Z_p) =(p-1)(m+n-2p-2)$.
\item (Odd skew) ${\rm lcdef}(Z_p) - {\rm lcdef}_{\rm gen}(Z_p) = 2(p-1)(2(m-p-1) +1)$.

\item (Even skew) ${\rm lcdef}(Z_p) - {\rm lcdef}_{\rm gen}(Z_p) = 4(p-1)(m-p-1)$.

\item (Symmetric) ${\rm lcdef}(Z_p) - {\rm lcdef}_{\rm gen}(Z_p) = \begin{cases} (n-p-1)(p-2) & p  \text{ even } \\ (n-p-1)(p-3)& p \text{ odd and }\, p\geq 3\end{cases}$.
\end{enumerate}
\end{proposition}
\begin{proof} This follows immediately from the computations of ${\rm lcd}(X,Z_p)$ in \cite{RW}, which is as follows:
\begin{itemize} \item (Generic) ${\rm lcd}(X,Z_p) = mn - (p+1)^2+1$,
\item (Odd skew) ${\rm lcd}(X,Z_p) = \binom{2m+1}{2} - \binom{2p+2}{2} +1$

\item (Even skew) ${\rm lcd}(X,Z_p) = \binom{2m}{2}-\binom{2p+2}{2} +1$

\item (Symmetric) ${\rm lcd}(X,Z_p) = \begin{cases} 1 + \binom{n+1}{2} - \binom{p+2}{2} & p \text{ even} \\1 + \binom{n}{2} - \binom{p+1}{2} & p \text{ odd}\end{cases}$.
\end{itemize}
The assertion follows by combining the above with \propositionref{lcdefgendet}.
\end{proof}

Using this, we can characterize which $Z_p$ are rational homology manifolds. Indeed, by \theoremref{thm-characterize}\eqref{thmmain} this is equivalent to ${\rm lcdef}(Z_p) = 0$ and ${\rm IC}_{Z_p} = \cohH^{c_p}_{Z_p}(\shO_X)$. The inequality ${\rm lcdef}_{\rm gen}(Z_p) \leq {\rm lcdef}(Z_p)$ shows that if ${\rm lcdef}(Z_p) =0$, then the difference ${\rm lcdef}(Z_p) - {\rm lcdef}_{\rm gen}(Z_p)$ is also 0, and so we can use the previous proposition to simplify our computations.

Recall that in the next proposition we assume $p\geq 1$ in cases (1)-(3) and we assume $p\geq 2$ in case (4).

\begin{proposition} \label{prop-nonRHMDeterminantal} In all cases (1)-(4), the variety $Z_p$ is not a rational homology manifold.
\end{proposition}
\begin{proof} For case (1), the difference ${\rm lcdef}(Z_p) - {\rm lcdef}_{\rm gen}(Z_p)$ and ${\rm lcdef}_{\rm gen}(Z_p)$ vanish if and only if $m+n = 2p+2$. As $p\leq n-1$, this gives $2p+2 \leq 2(n-1) +2 = 2n$, so this equality is only possible if $m = n$ and $p = n-1$. But for $m=n$, we have observed that $\cohH^{c_p}_{Z_p}(\shO_X) \neq {\rm IC}_{Z_p}$.

For case (2), we have that ${\rm lcdef}_{\rm gen}(Z_p) \geq 2 > 0$, so $Z_p$ can never be a rational homology manifold.

In case (3), ${\rm IC}_{Z_p} \neq \cohH^{c_p}_{Z_p}(\shO_X)$, so $Z_p$ cannot be a rational homology manifold.

For case (4), ${\rm lcdef}_{\rm gen}(Z_p) = 0$ if and only if $p = n-1$. But then ${\rm IC}_{Z_p} \neq \cohH^{c_p}_{Z_p}(\shO_X)$ as in this case, $p\not \equiv n \mod 2$.
\end{proof}

We conclude this subsection with an application of \theoremref{thm-ppbound} to give bounds on $\HRH(Z_p)$. To do this, we must compute ${\rm codim}_{Z_p}(Z_{p,{\rm nRS}})$. Note that $Z_{p,{\rm nRS}}$ is equal to $Z_{p-1}$ in cases (1)-(3) and $Z_{p-2}$ in case (4).

\begin{lemma}\label{lem-codnrs} In the notation above, we have the following formula for ${\rm codim}_{Z_p}(Z_{p,{\rm nRS}})$.

\begin{enumerate} \item (Generic) ${\rm codim}_{Z_p}(Z_{p,{\rm nRS}}) = m+n-2p+1$.

\item (Odd skew) ${\rm codim}_{Z_p}(Z_{p,{\rm nRS}}) = 4(m-p)+3$.

\item (Even skew) ${\rm codim}_{Z_p}(Z_{p,{\rm nRS}}) = 4(m-p)+1$.

\item (Symmetric) ${\rm codim}_{Z_p}(Z_{p,{\rm nRS}}) = 2(n-p)+3$ (when $p\geq 2$).
\end{enumerate}
\end{lemma}
\begin{proof} This is immediate by computing the differences
${\rm codim}_X(Z_{p-1}) - {\rm codim}_X(Z_p)$ in cases (1)-(3), and
${\rm codim}_{X}(Z_{p-2}) - {\rm codim}_{X}(Z_p)$ in case (4) using \corollaryref{cor-dim}.
\end{proof}

\begin{proposition} \label{prop-ComputeHRHDet} In the notation above, we have
\begin{enumerate} \item (Generic) $\HRH(Z_p) = 0$.

\item (Odd skew) $\HRH(Z_p) \in \{0,1\}$.

\item (Even skew) $\HRH(Z_p) \in \{0,1\}$.

\item (Symmetric) $\HRH(Z_p) \in \{0,1\}$ (when $p\geq 2$).
    
\end{enumerate}
    
\end{proposition}
\begin{proof} 

As $Z_p$ has rational singularities, we always have $\HRH(Z_p) \geq 0$, and we know $Z_p$ is not a rational homology manifold for all $p$, so $\HRH(Z_p) < \infty$.

By \theoremref{thm-ppbound}, we have inequality
\[ {\rm lcdef}_{\rm gen}(Z_p) \leq {\rm codim}_{Z_p}(Z_{p,{\rm nRS}}) - 2\HRH(Z_p) -3.\]
The assertion follows from \propositionref{lcdefgendet} and \lemmaref{lem-codnrs}.
\end{proof}

\begin{remark}
In the generic determinantal case, $\HRH(Z_p)=0$ also follows directly using \theoremref{thm-characterize}\eqref{thmmain} and \cite{Per}*{Cor. 1.4}. For this reason, we believe it will be interesting to carry out a study of the Hodge structures on the local cohomology modules analogous to \cite{Per} in the cases (2)-(4).
\end{remark}

\begin{remark} Some remarks are in order:
\begin{enumerate}
\item Note that equality holds in \theoremref{thm-ppbound} for generic determinantal varieties.
\item We observed earlier that the equality ${\rm lcdef}_{\textrm{gen}}(Z)= {\rm lcdef}(Z)$ holds when $Z_{\textrm{nRS}}$ is isolated. However, we see above that equality can also hold when $Z_{\textrm{nRS}}$ is non-isolated. Indeed, take for example $m=3, p=2$ in case (3) and apply \propositionref{prop-lcdefDifference}.
\item We also note that \propositionref{prop-lcdefDifference} gives many examples when ${\rm lcdef}_{\textrm{gen}}(Z)< {\rm lcdef}(Z)$. The smallest dimension of such $Z$ that we have in these classes of examples is 10 (generic determinantal with $m=4,n=3,p=2$).
\end{enumerate}
    
\end{remark}



\bibliography{bib}

\end{document}